\documentclass[a4paper,12pt]{article}
\usepackage{latexsym}
\usepackage{amsmath}
\usepackage{amssymb}
\usepackage{enumerate}
\usepackage{mathrsfs}

\usepackage{theorem}
\newtheorem{theorem}{Theorem}[section]
\newtheorem{proposition}[theorem]{Proposition}
\newtheorem{lemma}[theorem]{Lemma}
\newtheorem{corollary}[theorem]{Corollary}

\theorembodyfont{\rmfamily}
\newtheorem{remark}[theorem]{Remark}
\newtheorem{example}[theorem]{Example}
\newtheorem{definition}[theorem]{Definition}
\newtheorem{proof}{\textmd{\textit{Proof.}}}

\newtheorem{prooof}{\textmd{\textit{Proof of Proposition~$\ref{pr:Beem12.10}$.}}}

\makeatletter

\@addtoreset{equation}{section}
\makeatother

\newcommand{\qedd}{\hfill \square}
\newcommand{\ve}{\varepsilon}
\newcommand{\del}{\partial}
\newcommand{\lra}{\longrightarrow}
\newcommand{\e}{\mathrm{e}}

\newcommand{\N}{\ensuremath{\mathbb{N}}}

\newcommand{\R}{\ensuremath{\mathbb{R}}}
\newcommand{\Sph}{\ensuremath{\mathbb{S}}}

\newcommand{\fm}{\ensuremath{\mathfrak{m}}}
\newcommand{\bs}{\ensuremath{\mathbf{s}}}
\newcommand{\bK}{\ensuremath{\mathbf{K}}}

\newcommand{\sB}{\ensuremath{\mathsf{B}}}
\newcommand{\sD}{\ensuremath{\mathsf{D}}}
\newcommand{\sJ}{\ensuremath{\mathsf{J}}}
\newcommand{\sR}{\ensuremath{\mathsf{R}}}
\newcommand{\sT}{\ensuremath{\mathsf{T}}}

\def\vol{\mathop{\mathrm{vol}}\nolimits}
\def\diam{\mathop{\mathrm{diam}}\nolimits}

\def\Hess{\mathop{\mathrm{Hess}}\nolimits}
\def\Ric{\mathop{\mathrm{Ric}}\nolimits}
\def\trace{\mathop{\mathrm{trace}}\nolimits}

\newcommand{\wt}[1]{\widetilde{#1}}
\newcommand{\rev}[1]{\overleftarrow{#1}}

\title{Geometry of weighted Lorentz--Finsler manifolds~I: Singularity theorems}

\author{Yufeng LU\thanks{
Department of Mathematics, Osaka University, Osaka 560-0043, Japan
({\sf u105701g@ecs.osaka-u.ac.jp}, {\sf s.ohta@math.sci.osaka-u.ac.jp})}
\and
Ettore MINGUZZI\thanks{
Dipartimento di Matematica e Informatica ``U. Dini'', Universit\`a degli Studi di Firenze,
Via S.~Marta 3, I-50139 Firenze, Italy ({\sf ettore.minguzzi@unifi.it})}
\and
Shin-ichi OHTA\footnotemark[1] \textsuperscript{,}\thanks{
RIKEN Center for Advanced Intelligence Project (AIP),
1-4-1 Nihonbashi, Tokyo 103-0027, Japan}}

\date{}
\begin{document}

\maketitle

\begin{abstract}
We develop the theory of weighted Ricci curvature in a weighted Lorentz--Finsler framework
and extend the classical singularity theorems of general relativity.
In order to reach this result,
we generalize the Jacobi, Riccati and Raychaudhuri equations to weighted Finsler spacetimes
and study their implications for the existence of conjugate points along causal geodesics.
We also show a weighted Lorentz--Finsler version of the Bonnet--Myers theorem
based on a generalized Bishop inequality.
\end{abstract}

%\tableofcontents

\section{Introduction}%%%%%%%%%%%%%%%%%%%%%%
%%%%%%%%%%%%%%%%

The aim of this work is to develop the theory of weighted Ricci curvature
on \emph{weighted Lorentz--Finsler manifolds}
and show that the classical singularity theorems of general relativity \cite{HE}
can be generalized to this setting.
It is known that singularity theorems can be generalized to Finsler spacetimes \cite{AJ,Min-Ray},
and at least some of them have been generalized to the weighted Lorentzian framework
\cite{Ca,GW,WW1,WW2}
(we refer to \cite{GW,WW1} for some physical motivations in connection with the Brans--Dicke theory).
We will generalize many singularity theorems, including the classical ones
by Penrose, Hawking, and Hawking--Penrose, to the weighted Lorentz--Finsler setting.

By a \emph{weighted Lorentzian manifold} we mean a pair of a Lorentzian manifold $(M,g)$
and a weight function $\psi$ on $M$.
This is equivalent to considering a pair of $(M,g)$ and a measure $\fm$ on $M$
via the relation $\fm =\e^{-\psi} \,\vol_g$, where $\vol_g$ is the canonical volume measure of $g$.
The latter formulation was studied also in the Finsler framework \cite{Oint},
where the weight function associated with a measure needs to be a function
on the tangent bundle $TM \setminus \{0\}$
(since we do not have a unique canonical measure like $\vol_g$).
Motivated by these investigations, we work with an even more general structure,
namely a pair given by a Lorentz--Finsler spacetime $(M,L)$ and
a (positively $0$-homogeneous) function $\psi$ on the set of causal vectors.
In this framework we can include the unweighted case as well (as the case $\psi=0$),
while in general a constant function may not be associated with any measure.
See Section~\ref{sc:wRic} for a more detailed discussion.

Our results will be formulated with the \emph{weighted Ricci curvature} $\Ric_N$,
which is defined for $(M,L,\psi)$ in a similar way to the Finsler case \cite{Oint}
(see Definition~\ref{def:RicN}).
The real parameter $N$ is called the \emph{effective dimension} or the \emph{synthetic dimension}
(in connection with the synthetic theory of curvature-dimension condition, see below).
Our work not only unifies previous results,
but also improves previous findings already in the non-Finsler case,
particularly in dealing with the weight.
A weighted generalization of the Bishop inequality leads us
to a weighted Lorentz--Finsler version of the Bonnet--Myers theorem (Theorem~\ref{th:wBoMy}).
For what concerns singularity theorems,
we obtain not only the weighted Raychaudhuri equation,
but also the weighted Jacobi and Riccati equations (Section~\ref{sc:Ray}).
Moreover, we show that the genericity condition can be used in its classical formulation
(we need to introduce a weighted version as in \cite{Ca,WW2}
only in the extremal case of $N=0$, see Remarks~\ref{rm:t-gene}, \ref{rm:N=0}).
This fact simplifies the statements of some theorems.

Our results apply to every effective dimension,
$N \in (-\infty, 0] \cup [n,+\infty]$ in the timelike case
and $N \in (-\infty, 1] \cup [n,+\infty]$ in the null case.
The idea of including negative values of $N$ is recent, see \cite{WW1,WW2} for the Lorentzian case
(for us the spacetime dimension is $n+1$,
which means that the formulas in previous references have to undergo the replacements
$n \mapsto n+1$ and $N \mapsto N+1$ to be compared with our owns,
see Remark~\ref{rm:n+1}).
Our formulation of \emph{$\epsilon$-completeness} (Definitions~\ref{df:e-cplt}, \ref{df:ne-cplt}),
which is a key concept in singularity theorems,
generalizes that in \cite{WW1,WW2} and is very accurate:
We are able to identify a family of time parameters,
depending on a real variable $\epsilon$ belonging to an $\epsilon$-range dependent on $N$,
for which the incompleteness holds (see Propositions~\ref{bdo}, \ref{bdp}).
For $\epsilon=1$ one recovers the ordinary concept of completeness,
while for $\epsilon=0$ one recovers the $\psi$-completeness studied in \cite{WW1,WW2}.
Our $N$-dependent $\epsilon$-range explains why for $N \in [n,\infty)$
one can infer both (unweighted and weighted) forms of incompleteness,
while for negative $N$ one can infer only the $\psi$-incompleteness.

The investigation of weighted Lorentz--Finsler manifolds is meaningful
also from the view of synthetic studies of Lorentzian geometry.
This is motivated by the important breakthrough in the positive-definite case,
a characterization of the lower (weighted) Ricci curvature bound by the convexity of an entropy
in terms of optimal transport theory, called the \emph{curvature-dimension condition} CD$(K,N)$
(roughly speaking, CD$(K,N)$ is equivalent to $\Ric_N \ge K$).
We refer to \cite{CMS,LV,vRS,StI,StII,Vi} for the Riemannian case
and to \cite{Oint} for the Finsler case.
The curvature-dimension condition can be formulated in metric measure spaces
without differentiable structures.
Then one can successfully develop comparison geometry and geometric analysis
on such metric measure spaces.
Lorentzian counterparts of such a synthetic theory attracted growing interest recently,
see for instance \cite{AB,KuSa} for triangle comparison theorems,
\cite{BP,Br,EM,KeSu,Su} for optimal transport theory,
\cite{Mc} for a direct analogue to the curvature-dimension condition,
and \cite{MS} for an optimal transport interpretation of the Einstein equation.
We also refer to \cite{GKS,Min-causality}, the proceedings \cite{CGKM} and the references therein
for related investigations of less regular Lorentzian spaces.
Since the curvature-dimension condition is available both in Riemannian and Finsler manifolds,
it is important to know what kind of comparison geometric results
can be generalized to the Finsler setting.
Thus the results in this article will give some insights in the synthetic study of Lorentzian geometry.
We will continue the study of weighted Lorentz--Finsler manifolds
in a forthcoming paper on splitting theorems.

\subsection{The structure of singularity theorems} \label{com}%%%%%%%%%%%%%%%%%
%%%%%%%%%%%%%%%%%%%%%

Although our work will contain results whose scope exceeds that of singularity theorems,
it will be convenient to mention how singularity theorems are typically structured,
for that will clarify the focus of this work.

Singularity theorems are composed of the following three steps
(see \cite[Section~6.6]{Min-Rev} for further discussions):

\begin{enumerate}[I.]
\item\label{step1}
A non-causal statement assuming some form of geodesic completeness
plus some genericity and convergence conditions,
and implying the existence of conjugate points along geodesics or focal points
for certain (hyper)surfaces with special convergence properties,
e.g., our Corollaries~\ref{mci} and \ref{bhd}.
This step typically makes use of the Raychaudhuri equation.

\item\label{step2}
A non-causal statement to the effect that the presence of conjugate or focal points
spoils some length maximization property (achronal property in the null case),
for instance \cite[Proposition~5.1]{Min-Ray} will be used to show Proposition~\ref{vip}.

\item\label{step3}
A statement to the effect that under some causality conditions
as well as in presence of some special set (trapped set, Cauchy hypersurface)
the spacetime necessarily has a \emph{causal line} (a maximizing inextendible causal geodesic)
or a causal $S$-ray.
\end{enumerate}

The first two results go in contradiction with the last one,
so from here one infers the geodesic incompleteness.

%We say that $S \subset M$ is \emph{achronal} if $I^+(S) \cap S=\emptyset$
%(namely, no two points in $S$ are connected by a timelike curve).
%The set $S \subset M$ is called a \emph{trapped set}
%if the \emph{future horismos} $E^+(S):=J^+(S) \setminus I^+(S)$ of $S$ is nonempty and compact.
%Recall Definition~\ref{df:Cauchy} for the definition of Cauchy hypersurfaces.

Interestingly, the first two steps basically coincide for all the singularity theorems.
For instance, Penrose's and Gannon's singularity theorems \cite{Ga,P},
but also the topological censorship theorem \cite{FSW},
use the same versions of Steps~\ref{step1}, \ref{step2}.
Similarly, Hawking--Penrose's and Borde's singularity theorems \cite{Bo,HP}
use the same versions of Steps~\ref{step1}, \ref{step2}.
Most singularity theorems really differ just for the causality statement in Step~\ref{step3}.
For this reason, it is often convenient to identify the singularity theorem
with its \emph{causality core statement}, namely Step~\ref{step3}.
It turns out that this causality core statement in most cases involves just the cone distribution,
thereby it is fairly robust.

For instance,  we shall work with a Lorentz--Finsler space of Beem's type which is a special case
of a more general object called a \emph{locally Lipschitz proper Lorentz--Finsler space},
see \cite[Theorem~2.52]{Min-causality},
which is basically a distribution of closed cones $x \mapsto \overline{\Omega}_x$
plus a function $F:\overline{\Omega} \lra \R$ satisfying certain regularity properties.
For this structure and hence for our setting,
one can prove the following causality statement \cite[Theorem~2.67]{Min-causality}
(this result actually holds for more general \emph{closed cone structures}):
%\begin{quote}
{\em In a Finsler spacetime admitting a non-compact Cauchy hypersurface
 every nonempty compact set $S$ admits a future lightlike $S$-ray.}
%\end{quote}
(The various terms will be clarified in what follows.)
%We remark that $(M,L)$ is globally hyperbolic thanks to Proposition~\ref{pr:gl-hyp}.

There is also a simpler approach by which one can understand the validity of this type of causality core statements.
The local causality theory makes use of the existence of convex neighborhoods,
but does not make use of the curvature tensor.
The curvature tensor really makes its appearance only in Steps~\ref{step1} and \ref{step2} above.
Thus all the proofs of these causality core statements, being of topological nature,
pass through word-for-word from the Lorentzian to the Lorentz--Finsler case,
and since the weight is not used, to the weighted Lorentz--Finsler case.
These topological proofs can then be read from reviews of Lorentzian causality theory,
e.g., \cite[Theorem~6.23]{Min-Rev} includes the above statement.

%As an example of causality core statement that we shall meet in what follows we mention
%\begin{theorem} %\label{jvi}
%Let $(M,L)$ be a Finsler spacetime admitting a non-compact Cauchy hypersurface.
%Then every nonempty compact set $S$ admits a future lightlike $S$-ray.
%\end{theorem}

It is  important to understand that what we shall be  doing in the following sections
is to generalize Step~\ref{step1}.
Step~\ref{step2} has been already adapted to the Lorentz--Finsler setting in \cite{Min-Ray},
and hence to the weighted Lorentz--Finsler setting since it does not use the weight.
Step~\ref{step3} was also already generalized in \cite{Min-causality}
to frameworks broader than that of this work.
In this sense we are not considering the most general situation,
and we do not intend to make a full list of applications.
We wish to show that singularity theorems can be generalized to the weighted Lorentz--Finsler case,
by presenting several singularity theorems for the sake of illustrating the general strategy.
Once Steps~\ref{step1} and \ref{step2} are established,
by selecting a different causality core statement in Step~\ref{step3},
one can obtain other singularity theorems not explicitly considered in this article
(we refer to \cite[Section~8]{Min-Ray}, \cite[Section~6.6]{Min-Rev},
and \cite[Section~2.15]{Min-causality} for further singularity theorems
as well as more general statements).

\subsection{Notations and organization of the paper}%%%%%%%%%%%%%%%%%
%%%%%%%%%%%%%%%%%%%%%

Let us fix some terminologies and notations.
Riemannian and Finsler manifolds have positive-definite metrics.
The analogous structures in the Lorentzian signature will be called
Lorentzian manifolds and Lorentz--Finsler manifolds.
Lorentz--Finsler manifolds are also known as Lorentz--Finsler spaces
in other references, for example, \cite{Min-causality}.
The Lorentzian signature we use is $(-,+,\dots, +)$.
We stress that the dimension of the spacetime manifold is always $n+1$,
and the indices will be taken as $\alpha=0,1,\ldots,n$.

This article is organized as follows.
In Sections~\ref{sc:FLmfd} and \ref{sc:curv}, we introduce necessary notions of Finsler spacetimes,
including some causality conditions and the flag and Ricci curvatures.
We then introduce the weighted Ricci curvature in Section~\ref{sc:wRic}.
In Sections~\ref{sc:Ray} and \ref{sc:n-Ray},
we study the timelike and null Raychaudhuri equations, respectively,
which are applied in Section~\ref{sc:conj}
to investigate the existence of conjugate points along geodesics.
Finally, Section~\ref{sc:sing} is devoted to the proofs of some notable singularity theorems,
along the strategy outlined in Subsection~\ref{com} above.

\section{Finsler spacetimes}\label{sc:FLmfd}%%%%%%%%%%%%%%%%%
%%%%%%%%%%%%%%%%%%%%%

\subsection{Lorentz--Finsler manifolds}\label{ssc:FLmfd}%%%%%%%%%%%%%%%%%
%%%%%%%%%%%%%%%%%%%%%

Let $M$ be a connected $C^{\infty}$-manifold of dimension $n+1$ without boundary.
Given local coordinates $(x^\alpha)_{\alpha=0}^n$ on an open set $U \subset M$,
we will use the fiber-wise linear coordinates $(x^\alpha,v^\beta)_{\alpha,\beta=0}^n$ of $TU$
such that
\[ v=\sum_{\beta=0}^n v^\beta \frac{\del}{\del x^\beta}\Big|_x, \quad x \in U. \]
We employ Beem's definition of Lorentz--Finsler manifolds  \cite{Be}
(see Remark~\ref{rm:2cone} below for the relation with the other definitions).

\begin{definition}[Lorentz--Finsler structure]\label{df:FLstr}
A \emph{Lorentz--Finsler structure} of $M$ will be a function
$L\colon TM \lra \R$ satisfying the following conditions:
\begin{enumerate}[(1)]
\item \label{it:LF-1}
$L \in C^{\infty}(TM \setminus \{0\})$;

\item \label{it:LF-2}
$L(cv)=c^2 L(v)$ for all $v \in TM$ and $c>0$;

\item \label{it:LF-3}
For any $v \in TM \setminus \{0\}$, the symmetric matrix
\begin{equation}\label{eq:g_ij}
g_{\alpha \beta}(v) := \frac{\del^2 L}{\del v^\alpha \del v^\beta}(v),
 \quad \alpha,\beta=0,1,\ldots,n,
\end{equation}
is non-degenerate with signature $(-,+,\ldots,+)$.
\end{enumerate}
We will call $(M,L)$ a \emph{Lorentz--Finsler manifold} or a \emph{Lorentz--Finsler space}.
\end{definition}

We stress that the homogeneity condition \eqref{it:LF-2} is imposed only in the positive direction ($c>0$),
thus $L(-v) \neq L(v)$ is allowed.
We say that $L$ is \emph{reversible} if $L(-v)=L(v)$ for all $v \in TM$.
The matrix $(g_{\alpha \beta}(v))_{\alpha,\beta=0}^n$ in \eqref{eq:g_ij} defines
the Lorentzian metric $g_v$ of $T_xM$ by
\begin{equation}\label{eq:g_v}
g_v\bigg( \sum_{\alpha=0}^n a^\alpha \frac{\del}{\del x^\alpha}\Big|_x,
 \sum_{\beta=0}^n b^\beta \frac{\del}{\del x^\beta}\Big|_x \bigg)
 :=\sum_{\alpha,\beta=0}^n a^\alpha b^\beta g_{\alpha \beta}(v).
\end{equation}
By construction $g_v$ is the second order approximation of $2L$ at $v$.
Similarly to the positive-definite case,
the metric $g_v$ and Euler's homogeneous function theorem (see \cite[Theorem~1.2.1]{BCS})
will play a fundamental role in our argument.
We have for example
\[ g_v(v,v) =\sum_{\alpha,\beta=0}^n v^\alpha v^\beta g_{\alpha \beta}(v) =2L(v). \]

\begin{definition}[Timelike vectors]\label{df:time}
We call $v \in TM $ a \emph{timelike vector} if $L(v)<0$
and a \emph{null vector} if $L(v)=0$.
A vector $v$ is said to be \emph{lightlike} if it is null and nonzero.
The \emph{spacelike vectors} are those for which $L(v)>0$ or $v=0$.
The \emph{causal} (or \emph{non-spacelike}) vectors are those which are lightlike or timelike
($L(v) \le 0$ and $v \neq 0$).
The set of timelike vectors will be denoted by
\[ \Omega'_x:=\{ v \in T_xM \,|\, L(v)<0 \},
 \qquad \Omega' :=\bigcup_{x \in M} \Omega'_x. \]
\end{definition}

%We will later use $\Omega_x$ to denote a connected component of $\Omega'_x$.
%The vectors in $\Omega_x$ will be announced as the \emph{future-directed timelike vectors.}
Sometimes we shall make use of the function
$F: \overline{\Omega'} \lra [0,+\infty)$ defined by
\begin{equation}\label{eq:F}
F(v) := \sqrt{-g_v(v,v)} =\sqrt{-2L(v)},
\end{equation}
which measures the `length' of causal vectors.
The structure of the set of timelike vectors was studied in \cite{Be}.
We summarize fundamental properties in the next lemma,
see also \cite{Be,Per,Min-cone} for more detailed investigations.

\begin{lemma}[Properties of $\Omega'_x$]\label{lm:time}
Let $(M,L)$ be a Lorentz--Finsler manifold and $x \in M$.
\begin{enumerate}[{\rm (i)}]
\item \label{it:time-1}
We have $\Omega'_x \neq \emptyset$.

\item \label{it:time-2}
For each $c<0$, $T_xM \cap L^{-1}(c)$ is nonempty and positively curved
with respect to the linear structure of $T_xM$.

\item \label{it:time-3}
Every connected component of $\Omega'_x$ is a convex cone.
\end{enumerate}
\end{lemma}

\begin{proof}
\eqref{it:time-1}
If $L(v)>0$ for all $v \in T_xM \setminus \{0\}$,
then $T_xM \cap L^{-1}(1)$ is compact and $L$ is nonnegative-definite
at an extremal point of $T_xM \cap L^{-1}(1)$.
This contradicts Definition~\ref{df:FLstr}\eqref{it:LF-3}.
If $L \ge 0$ on $T_xM$ and there is $v \in T_xM \setminus \{0\}$ with $L(v)=0$,
then $L$ is again nonnegative-definite at $v$ and we have a contradiction.
Therefore we conclude $\Omega'_x \neq \emptyset$.

\eqref{it:time-2}
The first assertion $T_xM \cap L^{-1}(c) \neq \emptyset$ is straightforward from \eqref{it:time-1}
and the homogeneity of $L$.
The second assertion is shown by comparing $L$ and its second order approximation $g_v$
at $v \in T_xM \cap L^{-1}(c)$ (see \cite[Lemma~1]{Be}).

\eqref{it:time-3}
This is a consequence of \eqref{it:time-2}.
$\qedd$
\end{proof}

In the $2$-dimensional case ($n+1=2$),
the number of connected components of $\Omega'_x$ is not necessarily $2$,
even when $L$ is reversible (i.e., $L(-v)=L(v)$).

\begin{example}[Beem's example, \cite{Be}]\label{ex:Beem}
Let us consider the Euclidean plane $\R^2$.
Given $k \in \N$, we define $L:\R^2 \lra \R$ in the polar coordinates by $L(r,\theta) :=r^2 \cos k\theta$.
Then $\Hess L(r,\theta)$ has the negative determinant for $r>0$,
and the number of connected components of $\{ x \in \R^2 \,|\, L(x)<0 \}$ is $k$.
Notice that $L(r,\theta+\pi)=L(r,\theta)$ (reversible) if $k$ is even,
and $L(r,\theta+\pi)=-L(r,\theta)$ (non-reversible) if $k$ is odd.
\end{example}

This phenomenon could be regarded as a drawback of the formulation of Definition~\ref{df:FLstr}
from the view of theoretical physics, since it is difficult to interpret the causal structure
of such multi-cones (see \cite{Min-cone} for further discussions).
However, in the reversible case, it turned out that such an ill-posedness occurs only when $n+1=2$
(\cite[Theorem~7]{Min-cone}).

\begin{theorem}[Well-posedness for $n+1 \ge 3$]\label{th:cone}
Let $(M,L)$ be a reversible Lorentz--Finsler manifold of dimension $n+1 \ge 3$.
Then, for any $x \in M$, the set $\Omega_x'$ has exactly two connected components.
\end{theorem}

The key difference between $n+1=2$ and $n+1 \ge 3$ used in the proof is that
the sphere $\Sph^{n}$ is simply-connected if and only if $n \ge 2$ (see \cite[Theorem~6]{Min-cone}).
One may think of taking the product of $(\R^2,L)$ in Example~\ref{ex:Beem} and $\R$,
that is,
%\[
$\overline{L}(r,\theta,z):=r^2 \cos k\theta +z^2$.
%\]
This Lagrangian $\overline{L}$ is, however, twice differentiable at $(0,0,1)$ if and only if $k=2$.

\begin{remark}[Definitions of Lorentz--Finsler structures]\label{rm:2cone}
The analogue to Theorem~\ref{th:cone} in the non-reversible case is an open problem.
Nevertheless, it is in many cases acceptable to consider $L$ as defined just inside the future cone,
as in the approach by Asanov \cite{As}.
That is to say, we consider a smooth family of convex cones,
$\{\Omega_x\}_{x \in M}$ with $\Omega_x \subset T_xM \setminus \{0\}$,
and $L$ is defined only on $\bigcup_{x \in M} \overline{\Omega}_x$
such that $L<0$ on $\Omega_x$, $L=0$ on $\del \Omega_x$,
and having the Lorentzian signature
(studies of increasing functions for cone distributions can be found in \cite{FS,BS}
and their general causality theory is developed in \cite{Min-causality}).
In this case, under the natural assumption that $dL \neq 0$ on $\del\Omega_x$,
we can extend $L$ to $\wt{L}$ on the whole tangent bundle $TM$
such that the set of timelike vectors of $\wt{L}$ has exactly two connected components
in each tangent space (see \cite[Theorem~1]{Min-equiv}, $\wt{L}$ may not be reversible).
Therefore assuming that $L$ is globally defined as in Definition~\ref{df:FLstr} costs no generality.
Furthermore, in most arguments, given a (future-directed) timelike vector $v$,
we make use of $g_v$ from \eqref{eq:g_v} instead of $L$ itself.
\end{remark}

%\section{Causality theory and geodesics}\label{sc:causal}%%%%%%%%%%%%%%%%%
%%%%%%%%%%%%%%%%%%%%%%%%%

%This section is devoted to recalling necessary fundamental concepts in causality theory
%and the theory of (causal) geodesics, especially in globally hyperbolic manifolds.

\subsection{Causality theory}\label{ssc:causal}%%%%%%%%%%%%%

We recall some fundamental concepts in causality theory on a Lorentz--Finsler manifold $(M,L)$.
A continuous vector field $X$ on $M$ is said to be \emph{timelike}
if $L(X(x))<0$ for all $x \in M$.
If $(M,L)$ admits a timelike smooth vector field $X$,
then $(M,L)$ is said to be \emph{time oriented} by $X$, or simply \emph{time oriented}.
We will call a time oriented Lorentz--Finsler manifold a \emph{Finsler spacetime}.

A causal vector $v \in T_xM$ is said to be
\emph{future-directed} if it lies in the same connected component of
$\overline{\Omega'_x} \setminus \{0\}$ as $X(x)$.
We will denote by $\Omega_x \subset \Omega'_x$ the set of future-directed timelike vectors,
and set
\[ \Omega :=\bigcup_{x \in M} \Omega_x, \qquad
 \overline{\Omega} :=\bigcup_{x \in M} \overline{\Omega}_x, \qquad
 \overline{\Omega} \setminus \{0\} :=\bigcup_{x \in M} (\overline{\Omega}_x \setminus \{0\}). \]

A $C^1$-curve in $(M,L)$ is said to be timelike (resp.\ causal, lightlike, spacelike)
if its tangent vector is always timelike (resp.\ causal, lightlike, spacelike).
All causal curves will be future-directed in this article.
Given distinct points $x,y \in M$, we write $x \ll y$
if there is a future-directed timelike curve from $x$ to $y$.
Similarly, $x<y$ means that there is a future-directed causal curve from $x$ to $y$,
and $x \le y$ means that $x=y$ or $x<y$.

The \emph{chronological past}
and \emph{future} of $x$ are defined by
\[ I^-(x):=\{ y \in M \,|\, y \ll x\}, \qquad I^+(x):=\{ y \in M \,|\, x \ll y\}, \]
and the \emph{causal past} and \emph{future} are defined by
\[ J^-(x):=\{ y \in M \,|\, y \le x\}, \qquad J^+(x):=\{ y \in M \,|\, x \le y\}. \]
For a general set $S\subset M$, we define $I^-(S),I^+(S),J^-(S)$ and $J^+(S)$ analogously.

%Let us recall several causality conditions.

\begin{definition}[Causality conditions]\label{df:causal}
Let $(M,L)$ be a Finsler spacetime.
\begin{enumerate}[(1)]
\item $(M,L)$ is said to be \emph{chronological} if $x \notin I^+(x)$ for all $x\in M$.
\item We say that $(M,L)$ is \emph{causal} if there is no closed causal curve.
\item $(M,L)$ is said to be \emph{strongly causal} if, for all $x \in M$,
every neighborhood $U$ of $x$ contains another neighborhood $V$ of $x$
such that no causal curve intersects $V$ more than once.
\item We say that $(M,L)$ is \emph{globally hyperbolic}
if it is strongly causal and, for any $x,y \in M$, $J^+(x) \cap J^-(y)$ is compact.
\end{enumerate}
\end{definition}

Clearly strong causality implies causality, and a causal spacetime is chronological.
The chronological condition implies that the spacetime is non-compact.
The following concept plays an essential role
in the study of the geodesic incompleteness in general relativity.

\begin{definition}[Inextendibility]\label{df:inext} %[\protect{\cite[p. 61]{Be}}]
A future-directed causal curve $\eta:(a,b) \lra M$
is said to be \emph{future} (resp.\ \emph{past}) \emph{inextendible}
if $\eta(t)$ does not converge as $t \to b$ (resp.\ $t \to a$).
We say that $\eta$ is \emph{inextendible} if it is both future and past inextendible.
\end{definition}

Global hyperbolicity can be characterized in many ways.
Here we mention one of them in terms of Cauchy hypersurfaces
(see \cite[Proposition 6.12]{Min-Ray}, \cite[Theorem~1.3]{FS}).

\begin{definition}[Cauchy hypersurfaces]\label{df:Cauchy} %\protect{\cite[p. 65]{Be}}
A hypersurface $S \subset M$ is called a \emph{Cauchy hypersurface}
if every future-directed inextendible causal curve intersects $S$ exactly once.
\end{definition}

\begin{proposition}\label{pr:gl-hyp}
A Finsler spacetime $(M,L)$ is globally hyperbolic
if and only if it admits a smooth Cauchy hypersurface.
\end{proposition}

\subsection{Geodesics}\label{ssc:geod}%%%%%%%%%%
%%%%%%%%%%%%%%%%

Next we introduce some geometric concepts.
Define the \emph{Lorentz--Finsler length} of a piecewise $C^1$-causal curve
$\eta:[a,b] \lra M$ by (recall \eqref{eq:F} for the definition of $F$)
\[ \ell(\eta):=\int_a^b F\big( \dot{\eta}(t) \big) \,dt. \]
%In the Lorentzian signature we need to consider length \emph{maximizing} curves
%instead of length minimizing curves as done in Finsler geometry.
Then, for $x,y \in M$, we define the \emph{Lorentz--Finsler distance} $d(x,y)$ from $x$ to $y$ by
\[ d(x,y):=\sup_\eta \ell(\eta), \]
where $\eta$ runs over all piecewise $C^1$-causal curves from $x$ to $y$.
We set  $d(x,y):=0$  if there is no causal curve from $x$ to $y$.
We remark that, under the assumption of global hyperbolicity,
$d$ is finite and continuous (\cite[Proposition~6.8]{Min-Ray}).
A causal curve $\eta: I \lra M$ is said to be \emph{maximizing}
if, for every $t_1,t_2\in I$ with $t_1<t_2$, we have $d(\eta(t_1),\eta(t_2))= \ell(\eta|_{[t_1,t_2]})$.

%The \emph{action} of a $C^1$-causal curve $\eta:[a,b] \lra M$ is defined by
%\[ \mathcal{S}(\eta):=\int_a^b L \big( \dot{\eta}(t) \big) \,dt. \]
The \emph{Euler--Lagrange equation} for the \emph{action}
$\mathcal{S}(\eta):=\int_a^b L(\dot{\eta}(t)) \,dt$ provides the \emph{geodesic equation}
\begin{equation}\label{eq:geod}
\ddot{\eta}^\alpha +\sum_{\beta,\gamma=0}^n
 \widetilde{\Gamma}^\alpha_{\beta \gamma}(\dot{\eta})
 \dot{\eta}^\beta \dot{\eta}^\gamma =0,
\end{equation}
where we define
\begin{equation}\label{eq:gamma}
\widetilde{\Gamma}^\alpha_{\beta \gamma}(v)
:=\frac{1}{2} \sum_{\delta=0}^n g^{\alpha \delta}(v)
 \bigg( \frac{\del g_{\delta \gamma}}{\del x^\beta} +\frac{\del g_{\beta \delta}}{\del x^\gamma}
 -\frac{\del g_{\beta \gamma}}{\del x^\delta} \bigg)(v)
\end{equation}
for $v \in TM \setminus \{0\}$ and
$(g^{\alpha \beta}(v))$ denotes the inverse matrix of $(g_{\alpha \beta}(v))$.
%The equation \eqref{eq:geod} implies that $L(\dot{\eta})$ is constant.

%\begin{definition}[Geodesics]\label{df:geod}
We say that a $C^{\infty}$-causal curve $\eta:[a,b] \lra \R$ is \emph{geodesic}
if \eqref{eq:geod} holds for all $t \in (a,b)$.
%\end{definition}
%
Since $L(\dot{\eta})$ is constant by \eqref{eq:geod}, a causal geodesic is indeed
either a timelike geodesic or a lightlike geodesic.
%By the basic ODE theory, given arbitrary $v \in \overline{\Omega}_x$,
%there exists some $\ve>0$ and a unique $C^{\infty}$-geodesic
%$\eta:(-\ve,\ve) \lra M$ satisfying $\dot{\eta}(0)=v$.
%
%\begin{definition}[Exponential map]\label{eq:exp}
Given $v \in \overline{\Omega}_x$, if there is a geodesic $\eta:[0,1] \lra M$ with $\dot{\eta}(0)=v$,
then the \emph{exponential map} $\exp_x$ is defined by $\exp_x(v):=\eta(1)$.
%\end{definition}

Locally maximizing causal curves coincide with causal geodesics up to reparametrizations
(\cite[Theorem~6]{Min-conv}).
Under very weak differentiability assumptions on the metric,
this local maximization property can be used to define the notion of causal geodesics
(see \cite{Min-causality}).
We remark that, under Definition~\ref{df:FLstr}, due to a classical result by Whitehead,
the manifold admits convex neighborhoods.
% (this result has been generalized to rather weak differentiability conditions in \cite{Min-conv}).
Ultimately, this single fact makes it possible to work out much of causality theory
for Lorentz--Finsler manifolds in analogy with that for Lorentzian manifolds
(we refer to \cite{Min-conv,Min-Ray}).
%Our aim will be to show that many results, particularly the singularity theorems,
%can be further generalized to the weighted Lorentz--Finsler framework.

\section{Covariant derivatives and curvatures}\label{sc:curv}%%%%%%%%%%%%%%%%%
%%%%%%%%%%%%%%%%

In this section, along the argument in \cite[Chapter~6]{Shlec} (see also \cite{Onlga})
in the positive-definite case,
we introduce covariant derivatives (associated with the Chern connection)
and Jacobi fields by analyzing the behavior of geodesics.
Then we define the flag and Ricci curvatures in the spacetime context.
We refer to \cite[Section~2]{Min-Ray} for a further account.

Similarly to the previous section, $(M,L)$ will denote a Finsler spacetime
and all causal curves and vectors are future-directed.
In this section, however, this is merely for simplicity and the time-orientability plays no role.
Everything is local and can be readily generalized
to general causal vectors and geodesics on Lorentz--Finsler manifolds.

\subsection{Covariant derivatives}\label{ssc:covd}%%%%%%%%%
%%%%%%%%%%%%%%%%

We first introduce the coefficients of the \emph{geodesic spray}
and the \emph{nonlinear connection} as
\[ G^\alpha(v) :=\frac{1}{2} \sum_{\beta,\gamma=0}^n
 \widetilde{\Gamma}^\alpha_{\beta \gamma}(v) v^\beta v^\gamma, \qquad
N^\alpha_\beta(v) := \frac{\del G^\alpha}{\del v^\beta}(v) \]
for $v \in TM \setminus \{0\}$, and $G^\alpha(0)=N^\alpha_\beta(0):=0$.
Note that $G^\alpha$ is positively $2$-homogeneous
and $N^\alpha_\beta$ is positively $1$-homogeneous,
and $2 G^\alpha(v) =\sum_{\beta=0}^n N^\alpha_\beta(v)v^\beta$ holds
by the homogeneous function theorem.
The geodesic equation \eqref{eq:geod} is now written as
$\ddot{\eta}^\alpha +2 G^\alpha(\dot{\eta})=0$.
In order to define covariant derivatives,
we need to modify $\widetilde{\Gamma}^\alpha_{\beta \gamma}$ in \eqref{eq:gamma} as
%\begin{equation}\label{eq:Gamma}
\[ \Gamma^\alpha_{\beta \gamma}(v)
 :=\widetilde{\Gamma}^\alpha_{\beta \gamma}(v) -\frac{1}{2} \sum_{\delta,\mu=0}^n
 g^{\alpha \delta}(v) \bigg( \frac{\del g_{\delta \gamma}}{\del v^\mu} N^\mu_\beta
  +\frac{\del g_{\beta \delta}}{\del v^\mu} N^\mu_\gamma
 -\frac{\del g_{\beta \gamma}}{\del v^\mu} N^\mu_\delta \bigg)(v) \]
%\end{equation}
for $v \in TM \setminus \{0\}$.
Notice that these formulas are the same as those in \cite{Shlec}
(while $G^i(v)$ in \cite{Onlga} corresponds to $2G^{\alpha}(v)$ in this article).
%$\Gamma^\alpha_{\beta \gamma}$ coincides with that in \cite[(7)]{Min-Ray}.

\begin{definition}[Covariant derivatives]\label{df:cov}
For a $C^1$-vector field $X$ on $M$, $x \in M$ and $v,w \in T_xM$ with $w \neq 0$,
we define the \emph{covariant derivative} of $X$ by $v$ with reference (support) vector $w$ by
%\begin{equation}\label{eq:covd}
\[ D^w_v X :=\sum_{\alpha ,\beta=0}^n \bigg\{ v^\beta \frac{\del X^\alpha}{\del x^\beta}(x)
 +\sum_{\gamma=0}^n \Gamma^\alpha_{\beta \gamma}(w) v^\beta X^\gamma(x) \bigg\}
 \frac{\del}{\del x^\alpha} \Big|_x. \]
%\end{equation}
\end{definition}

The reference vector will be usually chosen as $w=v$ or $w=X(x)$.
The following result is shown in the same way as \cite[Section~6.2]{Shlec}
(see also \cite[Lemma~2.3]{Osplit}).

\begin{proposition}[Riemannian characterization]\label{pr:covd}
If $V$ is a nowhere vanishing $C^{\infty}$-vector field such that
%$V(x) \in \overline{\Omega}_x \setminus \{0\}$ for all $x \in M$ and
all integral curves of $V$ are geodesic, then we have
\[ D^V_V X =D^{g_V}_V X, \qquad D^V_X V =D^{g_V}_X V \]
for any differentiable vector field $X$, where $D^{g_V}$ denotes the covariant derivative
with respect to the Lorentzian structure $g_V$ induced from $V$ via \eqref{eq:g_v}.
\end{proposition}

Along a $C^{\infty}$-curve $\eta$ with $\dot{\eta} \neq 0$,
one can consider the covariant derivative along $\eta$,
%\begin{equation}\label{eq:covd+}
\[ D^{\dot{\eta}}_{\dot{\eta}} X(t) := \sum_{\alpha=0}^n \bigg( \dot{X}^\alpha
 +\sum_{\beta,\gamma=0}^n
 \Gamma^\alpha_{\beta \gamma}(\dot{\eta}) \dot{\eta}^\beta X^\gamma \bigg)(t)
 \,\frac{\del}{\del x^\alpha} \Big|_{\eta(t)}, \]
%\end{equation}
for vector fields $X$ along $\eta$,
where $X(t) =\sum_{\alpha=0}^n X^\alpha(t) (\del/\del x^\alpha)|_{\eta(t)}$.
Then the geodesic equation \eqref{eq:geod} coincides with $D^{\dot{\eta}}_{\dot{\eta}} \dot{\eta}=0$.
%By Proposition~\ref{pr:covd},
%the geodesic equation reads equivalently $D^{g_V}_{ \dot{\eta}} \dot \eta=0$,
%where $V$ is a vector field that extends $\dot\eta$ in a neighborhood of $\eta$.

For a nonconstant causal geodesic $\eta$
and $C^{\infty}$-vector fields $X,Y$ along $\eta$, we have
\begin{equation}\label{eq:g_eta}
\frac{d}{dt}\big[ g_{\dot{\eta}}(X,Y) \big]
 =g_{\dot{\eta}}(D^{\dot{\eta}}_{\dot{\eta}} X,Y) +g_{\dot{\eta}}(X,D^{\dot{\eta}}_{\dot{\eta}}Y)
\end{equation}
(see, e.g., \cite[Exercise~5.2.3]{BCS}).
One also has, for nowhere vanishing $X$,
\begin{equation}\label{eq:g_X}
\frac{d}{dt}\big[ g_X(X,Y) \big] =g_X(D^X_{\dot{\eta}} X,Y) +g_X(X,D^X_{\dot{\eta}}Y)
\end{equation}
(see \cite[Exercise~10.1.2]{BCS}).

\subsection{Jacobi fields}\label{ssc:Jacobi}%%%%%%%%%%%%%%%%%
%%%%%%%%%%%%%%%%

%In the Riemannian or Finsler context,
%Jacobi fields can be characterized as variational vector fields of geodesics.
%One can use this geometric intuition to define Jacobi fields
%(along \cite[Chapter~6]{Shlec}, see also \cite[Section~2]{Onlga}).
%We refer to \cite{Min-Ray} for a different derivation.

Next we introduce Jacobi fields.
Let $\zeta:[a,b] \times (-\ve,\ve) \lra M$ be a $C^{\infty}$-map
such that $\zeta(\cdot,s)$ is a causal geodesic for each $s \in (-\ve,\ve)$.
Put $\eta(t):=\zeta(t,0)$ and consider the variational vector field $Y(t) :=\del \zeta/\del s(t,0)$.
Then we have
\begin{align*}
&D^{\dot{\eta}}_{\dot{\eta}} D^{\dot{\eta}}_{\dot{\eta}} Y \\
&= \sum_{\alpha,\beta=0}^n \bigg\{ {-}2\frac{\del G^\alpha}{\del x^\beta}(\dot{\eta})
 +\sum_{\gamma=0}^n \bigg( \frac{\del N^\alpha_\beta}{\del x^\gamma}(\dot{\eta}) \dot{\eta}^\gamma
 -2\frac{\del N^\alpha_\beta}{\del v^\gamma}(\dot{\eta}) G^\gamma(\dot{\eta})
 + N^\alpha_\gamma(\dot{\eta}) N^\gamma_\beta(\dot{\eta}) \bigg) \bigg\} Y^\beta
 \frac{\del}{\del x^\alpha}\Big|_{\eta}.
\end{align*}
Now, we define
%\begin{equation}\label{eq:Rij}
\[ R^\alpha_\beta(v) :=2\frac{\del G^\alpha}{\del x^\beta}(v)
 -\sum_{\gamma=0}^n \bigg( \frac{\del N^\alpha_\beta}{\del x^\gamma}(v) v^\gamma
 -2\frac{\del N^\alpha_\beta}{\del v^\gamma}(v) G^\gamma(v) \bigg)
 -\sum_{\gamma=0}^n N^\alpha_\gamma(v) N^\gamma_\beta(v) \]
%\end{equation}
for $v \in \overline{\Omega}$ (note that $R^\alpha_\beta(0)=0$), and
\begin{equation}\label{eq:R_v}
R_v(w) :=\sum_{\alpha,\beta=0}^n R^\alpha_\beta(v) w^\beta \frac{\del}{\del x^\alpha}\Big|_x
\end{equation}
for $v \in \overline{\Omega}_x$ and $w \in T_xM$.
Then we arrive at the \emph{Jacobi equation}
\begin{equation}\label{eq:Jacobi}
D^{\dot{\eta}}_{\dot{\eta}} D^{\dot{\eta}}_{\dot{\eta}} Y +R_{\dot{\eta}}(Y) =0.
\end{equation}

\begin{definition}[Jacobi fields]\label{df:Jacobi}
A solution $Y$ to \eqref{eq:Jacobi} is called a \emph{Jacobi field} along a causal geodesic $\eta$.
\end{definition}

%Note that \eqref{eq:Jacobi} coincides with \cite[(16)]{Min-Ray}.
%By the ODE theory, given a causal geodesic $\eta:[a,b] \lra M$ and $v,w \in T_{\eta(a)}M$,
%there exists a unique Jacobi field $Y$ along $\eta$
%such that $Y(a)=v$ and $D^{\dot{\eta}}_{\dot{\eta}}Y(a)=w$.
We recall two important properties of $R_v$,
see \cite[Proposition~2.4]{Min-Ray} for a detailed account.

\begin{proposition}[Properties of $R_v$]\label{pr:R_v} %$\empty$
\begin{enumerate}[{\rm (i)}]
\item
We have $R_v(v)=0$ for every $v \in \overline{\Omega}_x$.
\item
$R_v$ is symmetric in the sense that
\begin{equation}\label{eq:symm}
g_v\big( w_1,R_v(w_2) \big) =g_v\big( R_v(w_1),w_2 \big)
\end{equation}
holds for all $v \in \overline{\Omega}_x \setminus \{0\}$ and $w_1,w_2 \in T_xM$.
\end{enumerate}
\end{proposition}

%\begin{definition}[Conjugate points]\label{df:conj}
Along a nonconstant causal geodesic $\eta:[a,b] \lra M$,
if there is a nontrivial Jacobi field $Y$ such that $Y(a) =Y(t) =0$ for some $t \in (a,b]$,
then we call $\eta(t)$ a \emph{conjugate point} of $\eta(a)$ along $\eta$.
%\end{definition}
%Equivalently, $\eta(t)$ is conjugate to $\eta(a)$ if
%$d(\exp_{\eta(a)})_{(t-a)\dot{\eta}(a)}$ does not have full rank.
The existence of conjugate points is a key issue throughout this article.

\subsection{Curvatures}\label{ssc:curv}%%%%%%%%%%%%%%%%%
%%%%%%%%%%%%%%%%

The flag and Ricci curvatures are defined by using $R_v$ in \eqref{eq:R_v} as follows.
The flag curvature corresponds to the sectional curvature in the Riemannian or Lorentzian context.

\begin{definition}[Flag curvature]\label{df:curv}
For $v \in {\Omega_x} $ and $w \in T_xM$ linearly independent of $v$,
define the \emph{flag curvature} of the $2$-plane $v \wedge w$ (a \emph{flag}) spanned by $v,w$
with \emph{flagpole} $v$ as
\begin{equation}\label{eq:flag}
\bK(v,w) :=-\frac{g_v(R_v(w),w)}{g_v(v,v) g_v(w,w) -g_v(v,w)^2}.
\end{equation}
\end{definition}

We remark that this is the opposite sign to \cite{BEE}, while the Ricci curvature will be the same.
The flag curvature $\bK(v,w)$ depends only on the $2$-plane $v \wedge w$
and the choice of the flagpole $\mathbb{R}_+v$ in it.

Note that, for $v$ timelike, the denominator in the right-hand side of \eqref{eq:flag} is negative.
The flag curvature is not defined for $v$ lightlike, for in this case the denominator could vanish.
Thus we define the Ricci curvature directly as the trace of $R_v$ in \eqref{eq:R_v}.

\begin{definition}[Ricci curvature]\label{df:Ric}
For $v \in \overline{\Omega}_x \setminus \{0\}$,
the \emph{Ricci curvature} or \emph{Ricci scalar} is defined as the trace of $R_v$,
i.e., $\Ric(v) :=\trace(R_{v})$.
\end{definition}

Since $\Ric(v)$ is positively $2$-homogeneous,
we can set $\Ric(0):=0$ by continuity.
We say that $\Ric \ge K$ holds \emph{in timelike directions} for some $K \in \R$
if we have $\Ric(v) \ge KF(v)^2 =-2KL(v)$ for all $v \in \Omega$.
For $v$ lightlike, since $L(v)=0$, only the nonnegative curvature condition $\Ric(v) \ge 0$
makes sense.

For a normalized timelike vector $v \in \Omega_x$ with $F(v)=1$,
$\Ric(v)$ can be given as $\Ric(v) =\sum_{i=1}^{n} \bK(v,e_i)$,
where $ \{v\} \cup \{e_i\}_{i=1}^{n}$ is an orthonormal basis with respect to $g_v$,
i.e., $g_v(e_i,e_j)=\delta_{ij}$ and $g_v(v,e_i)=0$ for all $i,j=1,\ldots,n$.

We deduce from Proposition~\ref{pr:covd} the following important feature of the Finsler curvature.
This is one of the main driving forces behind the recent developments
of comparison geometry and geometric analysis on Finsler manifolds
(see \cite{Shlec,Oint,Onlga}).

\begin{theorem}[Riemannian characterizations]\label{th:curv} %$\empty$ \\
Given a timelike vector $v \in \Omega_x$,
take a $C^1$-vector field $V$ on a neighborhood of $x$
such that $V(x)=v$ and every integral curve of $V$ is geodesic.
Then, for any $w \in T_xM$ linearly independent of $v$,
the flag curvature $\bK(v,w)$ coincides with the sectional curvature of
$v \wedge w$ for the Lorentzian metric $g_V$.
Similarly, the Ricci curvature $\Ric(v)$ coincides with the Ricci curvature of $v$ for $g_V$.
\end{theorem}

\begin{proof}
Let $\eta:(-\delta,\delta) \lra M$ be the geodesic with $\dot{\eta}(0)=v$
and observe that $V(\eta(t))=\dot{\eta}(t)$ by the condition imposed on $V$.
Take a $C^{\infty}$-variation $\zeta:(-\delta,\delta) \times (-\ve,\ve) \lra M$
of $\eta$ such that $\del_s \zeta(0,0)=w$
and that each $\zeta(\cdot,s)$ is an integral curve of $V$.
Then by the hypothesis, $\zeta(\cdot,s)$ is geodesic for all $s$
and hence $Y(t):=\del_s \zeta(t,0)$ is a Jacobi field along $\eta$.
Hence we deduce from the Jacobi equation \eqref{eq:Jacobi} that
\[ \bK(v,w)
 =-\frac{g_v(R_v(w),w)}{g_v(v,v)g_v(w,w)-g_v(v,w)^2}
 =\frac{g_v(D_{\dot{\eta}}^{\dot{\eta}} D_{\dot{\eta}}^{\dot{\eta}} Y(0),w)}{g_v(v,v)g_v(w,w)-g_v(v,w)^2}. \]

Now we compare this observation with the Lorentzian counterpart for $g_V$.
Since $\zeta$ is also a geodesic variation for $g_V$ (by Proposition~\ref{pr:covd}),
$Y$ is a Jacobi field also for $g_V$.
Moreover, it follows from Proposition~\ref{pr:covd} that
$D_{\dot{\eta}}^{\dot{\eta}} D_{\dot{\eta}}^{\dot{\eta}} Y(0)
 =D_{\dot{\eta}}^{g_V} D_{\dot{\eta}}^{g_V} Y(0)$.
This shows the first assertion, and the second assertion is obtained by taking the trace.
$\qedd$
\end{proof}

This observation is particularly helpful when we consider comparison theorems,
see Subsection~\ref{ssc:w-comp} for some instances.

\section{Weighted Ricci curvature}\label{sc:wRic}%%%%%%%%%%
%%%%%%%%%%%%%%%%%%%%%%%%%%%%%%

In this section we introduce the main ingredient of our results,
the \emph{weighted Ricci curvature},
for a triple $(M,L,\psi)$ where $(M,L)$ is a Finsler spacetime
and $\psi:\overline{\Omega} \setminus \{0\} \lra \R$
is a \emph{weight function} which is $C^{\infty}$ and positively $0$-homogeneous,
i.e., $\psi(cv)=\psi(v)$ for all $c>0$.

Let $\pi :\overline{\Omega} \setminus \{0\} \lra M$ be the bundle of causal vectors.
The function $\psi$ can be used to define a section of the pullback bundle
$\pi^*[\bigwedge^{n+1}(T^*M)] \lra \overline{\Omega} \setminus \{0\}$ as
\[ \Phi(x,v) \,dx^0 \wedge dx^1 \wedge \cdots \wedge dx^n, \qquad
%where we used local coordinates $(x^\alpha)_{\alpha=0}^n$ and
 \Phi(x,v):=\e^{-\psi(v)}  \sqrt{-{\det}\big[ \big( g_{\alpha \beta}(v) \big)_{\alpha,\beta=0}^n \big]}, \]
provided that $M$ is orientable.
In other words, we can consider a similar formula (even when $M$ is not orientable) as follows:
For every causal vector field $V$ on $M$,
\[ \fm_V(dx) :=\Phi\big( x,V(x) \big)\,dx^0 dx^1 \cdots dx^n
 = \e^{-\psi(V(x))} \,{\vol}_{g_V}(dx) \]
defines a measure $\fm_V$ on $M$,
where $\vol_{g_V}$ is the volume measure induced from $g_V$.

This structure $(M,L,\psi)$ generalizes that of a \emph{Lorentz--Finsler measure space},
which means a triple $(M,L,\fm)$ where $\fm$ is a positive $C^{\infty}$-measure on $M$
in the sense that, in each local coordinates $(x^\alpha)_{\alpha=0}^n$,
$\fm$ is written as $\fm(dx) =\Phi(x) \,dx^0 dx^1 \cdots dx^n$
(see \cite{Oint} for the positive-definite case).
In this setting the function $\psi$ is defined so as to satisfy,
for $v \in \overline{\Omega}_x \setminus \{0\}$,
%\begin{equation} \label{jus}
\[ \Phi(x)= \e^{-\psi(v)} \sqrt{-{\det}\big[ \big( g_{\alpha \beta}(v) \big)_{\alpha,\beta=0}^n \big]}. \]
%\end{equation}
%Thus $\fm_V$ coincides with $\fm$ and we are back to the structure of
%a Lorentz--Finsler measure space.
Notice that $g_{\alpha\beta}(v)$ depends on the direction $v$ in the Lorentz--Finsler (or Finsler) case.
%Hence, when one is interested in a Lorentz--Finsler measure space,
%it is natural to begin with an arbitrary measure
%because there is no unique way of choosing a canonical measure.
This is the reason why we consider a function on $\overline{\Omega} \setminus \{0\}$,
instead of a function on $M$ as in the Lorentzian case.
%We refer to \cite{Oint,ORand} for further discussions in the positive-definite case.
Our approach here, considering a general function $\psi$ not necessarily induced from a measure,
represents a further generalization which allows us to identify the \emph{unweighted case}:
We shall say that we are in the unweighted case if $\psi$ is constant.
(There may not exist any measure such that $\psi$ is constant,
see \cite{ORand} for a related study in the positive-definite case.)
Since all the following calculations involve only the derivatives of $\psi$,
we can regard the choice $\psi=0$ as the only unweighted case.

We need to modify $\Ric(v)$ defined in Definition~\ref{df:Ric} according to the choice of $\psi$,
so as to generalize the definition of \cite{Oint}  for  the Finsler measure space case.
As a matter of notation, given a causal geodesic $\eta(t)$ we shall write
\begin{equation}\label{eq:psi_eta}
\psi_\eta(t) :=\psi \big( \dot\eta(t) \big).
\end{equation}

\begin{definition}[Weighted Ricci curvature]\label{def:RicN} %$\empty$ \\
On $(M,L, \psi)$ with $\dim M=n+1$,
given a nonzero causal vector $v \in \overline{\Omega}_x \setminus \{0\}$,
let $\eta:(-\ve,\ve) \lra M$ be the geodesic with $\dot{\eta}(0)=v$.
Then, for $N \in \R \setminus \{n\}$,
we define the \emph{weighted Ricci curvature} by
\begin{equation}\label{eq:Ric_N}
\Ric_N(v) :=\Ric(v) +\psi_\eta''(0) -\frac{\psi_\eta'(0)^2}{N-n}.
\end{equation}
As the limits of $N \to +\infty$ and $N \downarrow n$, we also define
\[
 \Ric_\infty(v) :=\Ric(v)+  \psi_\eta''(0) , \qquad
 \Ric_n(v) :=\begin{cases}
 \Ric(v)+    \psi_\eta''(0)  & \textrm{ if }\psi_\eta'(0)=0,\\
 -\infty & \textrm{ if }\psi_\eta'(0) \neq 0. %\label{mvi}
 \end{cases}
\]
\end{definition}

\begin{remark}\label{rm:n+1}
Because of our notation $\dim M=n+1$,
$\Ric_N$ in this article corresponds to $\Ric_{N+1}$ in \cite{Oint,Onlga}
or $\Ric_f^{N+1}$ in \cite{WW1,WW2}.
\end{remark}

Similarly to Definition~\ref{df:Ric},
we say that $\Ric_N \ge K$ holds in timelike directions for some $K \in \R$
if we have $\Ric_N(v) \ge KF(v)^2$ for all $v \in \Omega$,
and $\Ric_N \ge 0$ in null directions means that $\Ric_N(v) \ge 0$
for all lightlike vectors $v$.

The weighted Ricci curvature $\Ric_N$ is also called the \emph{Bakry--\'Emery--Ricci curvature},
due to the pioneering work by Bakry--\'Emery \cite{BE} in the Riemannian situation
(we refer to the book \cite{BGL} for further information).
The Finsler version was introduced in \cite{Oint} as we mentioned,
and we refer to \cite{Ca} for the case of Lorentzian manifolds.
%As we remarked above,
%a fundamental difference between the Lorentzian and Lorentz--Finsler settings
%is that $\psi$ is a function on $M$ in the former,
%whereas it is a function on $\overline{\Omega} \setminus \{0\}$ in the latter.

\begin{remark}[Remarks on $\Ric_N$]\label{rm:RicN} %$\empty$
\begin{enumerate}[(a)]
\item
In the unweighted case, we have $ \Ric_N(v)=\Ric(v)$ for every $N\in (-\infty,+\infty]$.
In general, it is clear by definition that $\Ric_N$ is monotone non-decreasing
in the ranges $[n,+\infty]$ and $(-\infty,n)$, and we have
\begin{equation}\label{eq:mono}
\Ric_n(v) \le \Ric_N(v) \le \Ric_{\infty}(v) \le \Ric_{N'}(v)
\end{equation}
for any $N \in (n,+\infty)$ and $N'\in(-\infty,n)$.

\item
The study of the case where $N \in (-\infty,n)$ is rather recent.
The above monotonicity in $N$ implies that $\Ric_N \ge K$ with $N<n$
is a weaker condition than $\Ric_{\infty} \ge K$.
Nevertheless, one can generalize a number of results to this setting,
see \cite{KM,Mil-neg,Oneg,Wy} for the positive-definite case
and \cite{WW1,WW2} for the Lorentzian case.

\item
The Riemannian characterization as in Theorem~\ref{th:curv} is valid
also for the weighted Ricci curvature.
Take a $C^1$-vector field $V$ such that $V(x)=v$
and all integral curves of $V$ are geodesic.
Then $V$ induces the metric $g_V$ and the weight function $\psi_V:=\psi(V)$
on a neighborhood of $x$,
thus we can calculate the weighted Ricci curvature $\Ric_N^{(g_V,\psi_V)}(v)$
for $(M,g_V, \psi_V)$.
Since $\eta$ is geodesic also for $g_V$ and $\dot{\eta}(t)=V(\eta(t))$ by construction,
we deduce that $\Ric_N(v)$ in \eqref{eq:Ric_N} coincides with the Lorentzian counterpart:
\[
\Ric_N^{(g_V,\psi_V)}(v) =\Ric^{g_V}(v) +(\psi_V \circ \eta)''(0) -\frac{(\psi_V \circ \eta)'(0)^2}{N-n}.
\]
\end{enumerate}
\end{remark}

\section{Weighted Raychaudhuri equation}\label{sc:Ray}%%%%%%%%%
%%%%%%%%%%%%%%%%%%%%%%%%%%%%%%

Next we consider the \emph{Raychaudhuri equation} on weighted Finsler spacetimes.
In the unweighted case, the Finsler Raychaudhuri equation was established in \cite{Min-Ray}
along with corresponding singularity theorems.
Our approach is inspired by \cite{Ca} on the weighted Lorentzian setting.
(A counterpart to the Raychaudhuri equation in the positive-definite case
is the Bochner--Weitzenb\"ock formula; for that we refer to \cite{OSbw} in the Finsler context.)

\subsection{Weighted Jacobi and Riccati equations}\label{ssc:w-Ricc}%%%%%%%%%
%%%%%%%%%%%%%%%%%%%%%%%%%%%%%%

We begin with the notion of Jacobi and Lagrange tensor fields.
%(these were mentioned in the proof of Theorem~\ref{th:BoMy}).
We say that a timelike geodesic $\eta$ has \emph{unit speed}
if $F(\dot{\eta}) \equiv 1$ ($L(\dot{\eta}) \equiv -1/2$).
For simplicity, the covariant derivative of a vector field $X$ along $\eta$
will be denoted by $X'$. Observe that this time-differentiation by acting linearly passes to the tensor bundle over $\eta$
and in particular to endomorphisms $E$ as $E'(P):=E(P)'-E(P')$.
%As in the proof of Theorem~\ref{th:BoMy},
We denote by $N_{\eta}(t) \subset T_{\eta(t)}M$
the $n$-dimensional subspace $g_{\dot{\eta}(t)}$-orthogonal to $\dot\eta(t)$.

\begin{definition}[Jacobi, Lagrange tensor fields]\label{df:Jtensor}
Let $\eta: I \lra M$ be a timelike geodesic of unit speed.
\begin{enumerate}[(1)]
\item
A smooth tensor field $\sJ$, giving an endomorphism
$\sJ(t):N_{\eta}(t) \lra N_{\eta}(t)$ for each $t \in I$,
is called a \emph{Jacobi tensor field along $\eta$} if we have
\begin{equation}\label{eq:Jtensor}
\sJ''+\sR \sJ=0
\end{equation}
and $\ker(\sJ(t)) \cap \ker(\sJ'(t)) =\{0\}$ holds for all $t \in I$,
where $\sR(t):=R_{\dot\eta(t)}:N_{\eta}(t) \lra N_{\eta}(t)$ is the curvature endomorphism.

\item
A Jacobi tensor field $\sJ$ is called a \emph{Lagrange tensor field} if
\begin{equation}\label{eq:Ltensor}
(\sJ')^{\sT} \sJ -\sJ^{\sT} \sJ'=0
\end{equation}
holds on $I$, where the transpose $\sT$ is taken with respect to $g_{\dot \eta}$.
\end{enumerate}
\end{definition}

For $t \in I$ where $\sJ(t)$ is invertible, note that \eqref{eq:Ltensor} is
equivalent to the $g_{\dot\eta}$-symmetry of $\sJ'\sJ^{-1}$.
At those points we define $\sB:= \sJ'\sJ^{-1}$.
Then, multiplying \eqref{eq:Jtensor} by $\sJ^{-1}$ from right, we arrive at the \emph{Riccati equation}
\begin{equation}\label{eq:Ricca}
\sB' +\sB^2 +\sR=0.
\end{equation}
For thoroughness, let us explain the precise meaning of \eqref{eq:Jtensor}
and \eqref{eq:Ltensor}.

\begin{remark}\label{rm:Jtensor}
\begin{enumerate}[(a)]
\item
The equation \eqref{eq:Jtensor} means that, for any $g_{\dot{\eta}}$-parallel vector field
$P(t) \in N_{\eta}(t)$ along $\eta$ (namely $P' \equiv 0$),
$Y(t):=\sJ(t)(P(t))$ is a Jacobi field along $\eta$ such that $g_{\dot{\eta}}(Y,\dot \eta)=0$.
Then we find from \eqref{eq:g_eta} that $g_{\dot{\eta}}(Y',\dot \eta)=0$.
Thus we have $\sJ':N_{\eta}(t) \lra N_{\eta}(t)$ and $\sB:N_{\eta}(t) \lra N_{\eta}(t)$.

\item
Proposition~\ref{pr:R_v} ensures $R_{\dot{\eta}(t)}(v) \in N_{\eta}(t)$ for all $v \in T_{\eta(t)}M$.
The $g_{\dot\eta}$-symmetry in \eqref{eq:Ltensor} means that,
given two $g_{\dot{\eta}}$-parallel vector fields $P_1(t),P_2(t) \in N_{\eta}(t)$ along $\eta$,
the Jacobi fields $Y_i:=\sJ(P_i)$ satisfy
\begin{equation}\label{eq:Ltensor'}
g_{\dot{\eta}}(Y'_1,Y_2) -g_{\dot{\eta}}(Y_1,Y'_2) =0
\end{equation}
on $I$.
Since \eqref{eq:Jtensor} and \eqref{eq:symm} imply
$[g_{\dot{\eta}}(Y'_1,Y_2) -g_{\dot{\eta}}(Y_1,Y'_2)]' \equiv 0$,
we have \eqref{eq:Ltensor'} for all $t$ if it holds at some $t$. 
\end{enumerate}
\end{remark}

We introduce fundamental quantities in the analysis of Jacobi tensor fields
along the Lorentz--Finsler treatment of \cite{Min-Ray}.

\begin{definition}[Expansion, shear tensor]\label{df:expan}
Let $\sJ$ be a Jacobi tensor field along a timelike geodesic $\eta:I \lra M$ of unit speed.
For $t \in I$ where $\sJ(t)$ is invertible, we define 
%$\sB:= \sJ'\sJ^{-1}$ (as in \eqref{eq:B}).
%Then we define 
the \emph{expansion scalar} by
\[ \theta(t) :=\trace \big( \sB(t) \big) \]
and the \emph{shear tensor} (the traceless part of $\sB$) by
\[ \sigma(t) :=\sB(t) -\frac{\theta(t)}{n}I_{n}, \]
where $I_{n}$ represents the identity of $N_{\eta}(t)$.
\end{definition}

We proceed to the weighted situation.
Recall that $\psi$ is a function on $\overline{\Omega} \setminus \{0\}$
and, along a causal geodesic $\eta$, we set $\psi_{\eta}:=\psi(\dot{\eta})$ (see \eqref{eq:psi_eta}).
For a Jacobi tensor field $\sJ$ along a timelike geodesic $\eta:I \lra M$ of unit speed,
define the {\em weighted Jacobi endomorphism} by
\begin{equation}\label{eq:J_psi}
\sJ_\psi(t):=\e^{-\psi_\eta(t)/n} \sJ(t).
\end{equation}
Now we introduce an auxiliary time, the \emph{$\epsilon$-proper time}, defined by
\begin{equation} \label{par}
\tau_\epsilon:=\int \e^{\frac{2(\epsilon-1)}{n}\psi_\eta(t)} \,dt,
\end{equation}
where $t$ is the usual proper time parametrization.
Notice that $\tau_\epsilon$ coincides with the usual proper time for $\epsilon=1$,
and the case of $\epsilon=0$ was introduced in \cite{WW1}.
For brevity the (covariant) derivative in $\tau_\epsilon$ will be denoted by $*$.
For instance,
\[ \eta^*(t) :=\frac{d[\eta \circ \tau_{\epsilon}^{-1}]}{d\tau_{\epsilon}}\big( \tau_{\epsilon}(t) \big)
 =\e^{\frac{2(1-\epsilon)}{n}\psi_{\eta}(t)} \dot{\eta}(t). \]
Let us also introduce a weighted counterpart to the \emph{curvature endomorphism}:
\begin{equation}\label{mfi}
\sR_{(N,\epsilon)}(t) := \e^{\frac{4(1-\epsilon)}{n} \psi_\eta(t)}
 \left\{\sR(t)+\frac{1}{n}\bigg(\psi_\eta''(t)-\frac{\psi'_{\eta}(t)^2}{N-n}\bigg) I_{n}\right\}
\end{equation}
for $N \neq n$ (compare this with $R_f(t)$ in \cite[Definition~2.7]{Ca}).
This expression is chosen in such a way that
\[ \trace(\sR_{(N,\epsilon)})
 =\e^{\frac{4(1-\epsilon)}{n} \psi_{\eta}} \Ric_N(\dot \eta)=\Ric_N(\eta^*). \]
A straightforward calculation shows the following \emph{weighted Jacobi equation},
which generalizes \eqref{eq:Jtensor}.

\begin{lemma}[Weighted Jacobi equation]\label{lm:w-Jacobi}
With the notations as above, we have
\begin{equation}\label{eq:w-Jacobi}
\sJ_\psi^{**}+\frac{2\epsilon}{n} \psi_\eta^* \sJ_\psi^*+ \sR_{(0,\epsilon)} \sJ_\psi=0.
\end{equation}
\end{lemma}

\begin{proof}
Recalling the definition of $\sJ_\psi$ in \eqref{eq:J_psi}, we observe
\begin{align}
\sJ_\psi^*
&= \e^{-\psi_{\eta}/n}
 \bigg( \e^{\frac{2(1-\epsilon)}{n}\psi_\eta} \sJ' -\frac{\psi_\eta^*}{n}\sJ \bigg), \label{eq:J^*}\\
\sJ_\psi^{**}
&= \e^{-\psi_{\eta}/n} \bigg\{ \e^{\frac{4(1-\epsilon)}{n}\psi_\eta} \sJ''
 +\frac{1-2\epsilon}{n}\psi_\eta^* \e^{\frac{2(1-\epsilon)}{n}\psi_\eta} \sJ'
 -\frac{\psi_\eta^{**}}{n} \sJ
 -\frac{\psi_\eta^*}{n}
 \bigg( \e^{\frac{2(1-\epsilon)}{n}\psi_\eta} \sJ' -\frac{\psi_\eta^*}{n}\sJ \bigg) \bigg\} \nonumber\\
&= \e^{-\psi_{\eta}/n} \bigg( \e^{\frac{4(1-\epsilon)}{n}\psi_\eta} \sJ''
 -\frac{2\epsilon}{n}\psi_\eta^* \e^{\frac{2(1-\epsilon)}{n}\psi_\eta} \sJ'
 -\frac{\psi_\eta^{**}}{n} \sJ +\frac{(\psi_\eta^*)^2}{n^2} \sJ \bigg). \nonumber
\end{align}
Moreover,
\begin{equation}\label{eq:Rin*}
\sR_{(0,\epsilon)}
 =\e^{\frac{4(1-\epsilon)}{n}\psi_\eta} \sR
 +\frac{1}{n} \bigg( \psi_\eta^{**} -\frac{2(1-\epsilon)}{n}(\psi_\eta^*)^2
 +\frac{(\psi_\eta^*)^2}{n} \bigg) I_n.
\end{equation}
Therefore we have, with the help of $\sJ'' +\sR \sJ=0$ in \eqref{eq:Jtensor},
\[ \sJ_\psi^{**} +\sR_{(0,\epsilon)}\sJ_\psi
 =-\e^{-\psi_\eta/n} \bigg( \frac{2\epsilon}{n}\psi_\eta^* \e^{\frac{2(1-\epsilon)}{n}\psi_\eta} \sJ'
 -\frac{2\epsilon}{n^2} (\psi_\eta^*)^2 \sJ \bigg)
 =-\frac{2\epsilon}{n} \psi_\eta^* \sJ_\psi^*. \]
$\qedd$
\end{proof}

For $t \in I$ where $\sJ(t)$ is invertible, we define
\begin{equation}\label{eq:w-B}
\sB_{\epsilon}(t) :=\sJ_{\psi}^*(t) \sJ_{\psi}(t)^{-1}
 =\e^{\frac{2(1-\epsilon)}{n}\psi_\eta(t)} \sB(t) -\frac{\psi_{\eta}^*(t)}{n} I_{n},
\end{equation}
%corresponding to \eqref{eq:B},
where we used \eqref{eq:J^*} and suppressed the dependence on $\psi$.
Similarly to Lemma~\ref{lm:w-Jacobi} above,
one can show the \emph{weighted Riccati equation} generalizing \eqref{eq:Ricca} as follows.

\begin{lemma}[Weighted Riccati equation]\label{lm:w-Ricca}
With the notations as above, we have
\begin{equation}\label{eq:w-Ricc}
\sB_\epsilon^* +\frac{2\epsilon}{n}\psi_\eta^* \sB_\epsilon+\sB_\epsilon^2 +\sR_{(0,\epsilon)}=0.
\end{equation}
\end{lemma}

\begin{proof}
We deduce from
\[ \sB_\epsilon^* =\e^{\frac{4(1-\epsilon)}{n}\psi_\eta} \sB'
 +\frac{2(1-\epsilon)}{n}\psi_\eta^* \e^{\frac{2(1-\epsilon)}{n}\psi_\eta} \sB
 -\frac{\psi_\eta^{**}}{n}I_n, \]
\eqref{eq:Rin*} and $\sB' +\sB^2 +\sR=0$ in \eqref{eq:Ricca} that
\begin{align*}
\sB_\epsilon^* +\sB_\epsilon^2 +\sR_{(0,\epsilon)}
&= \frac{2(1-\epsilon)}{n}\psi_\eta^* \e^{\frac{2(1-\epsilon)}{n}\psi_\eta} \sB
 -2\e^{\frac{2(1-\epsilon)}{n} \psi_\eta} \frac{\psi_\eta^*}{n} \sB +\frac{(\psi_\eta^*)^2}{n^2} I_n
 -\frac{1-2\epsilon}{n^2}(\psi_\eta^*)^2 I_n \\
&= -\frac{2\epsilon}{n}\psi_\eta^* \e^{\frac{2(1-\epsilon)}{n}\psi_\eta} \sB
 +\frac{2\epsilon}{n^2} (\psi_\eta^*)^2 I_n
 = -\frac{2\epsilon}{n} \psi_\eta^* \sB_\epsilon.
\end{align*}
$\qedd$
\end{proof}

Observe that for $\epsilon=0$ both the weighted Jacobi and Riccati equations are simplified
to have the same forms as the unweighted situation
(compare this with \cite[Proposition~2.8]{Ca}, adding the factor $\e^{\frac{2(1-\epsilon)}{n}\psi_\eta}$
enabled us to remove the extra term appearing there).
We define the $\epsilon$-\emph{expansion scalar} by
\begin{equation}\label{eq:w-expan}
\theta_{\epsilon}(t) :={\trace}\big( \sB_{\epsilon}(t) \big)
 =\e^{\frac{2(1-\epsilon)}{n}\psi_\eta(t)} \theta(t) -\psi_{\eta}^*(t)
 =\e^{\frac{2(1-\epsilon)}{n}\psi_\eta(t)} \big( \theta(t) -\psi_{\eta}'(t) \big).
\end{equation}
For $\epsilon=0$, we may also write
$\theta_\psi :=\theta_0 =\e^{\frac{2}{n}\psi_\eta} (\theta -\psi'_{\eta})$.
Define the \emph{$\epsilon$-shear tensor} by
\begin{equation} \label{eq:w-shear}
\sigma_\epsilon(t) := \sB_{\epsilon}(t) -\frac{\theta_{\epsilon}(t)}{n}I_{n}
 =\e^{\frac{2(1-\epsilon)}{n}\psi_\eta(t)} \sigma(t).
\end{equation}
Since $\sB$ is $g_{\dot \eta}$-symmetric, so are $\sB_\epsilon$ and $\sigma_\epsilon$.

\subsection{Raychaudhuri equation}\label{ssc:Ray}%%%%%%%%
%%%%%%%%%%%%%%%%%%%

Taking the trace of the weighted Riccati equation \eqref{eq:w-Ricc},
we obtain the \emph{weighted Raychaudhuri equation}
displaying $\Ric_0$ and after a straightforward manipulation the versions displaying $\Ric_N$.

\begin{theorem}[Timelike weighted Raychaudhuri equation]\label{th:wRay} %$\empty$\\
Let $\sJ$ be a nonsingular Lagrange tensor field along a future-directed timelike geodesic
$\eta:I \lra M$ of unit speed.
Then, for $N=0$, the $\epsilon$-expansion $\theta_\epsilon$ satisfies
\begin{equation} \label{jsp}
\theta_\epsilon^* +\frac{2\epsilon}{n} \psi_\eta^* \theta_\epsilon +\frac{\theta_\epsilon^2}{n}
 +\trace(\sigma_\epsilon^2)+ \Ric_0(\eta^*)=0
\end{equation}
on $I$.
For $N\in (-\infty,+\infty)\backslash \{0,n\}$, $\theta_\epsilon$ satisfies
\begin{equation} \label{fsp}
\theta_\epsilon^*+ \left(1-\epsilon^2\frac{N-n}{N}\right) \frac{\theta_\epsilon^2}{n}
 +\frac{N (N-n)}{n} \left(\frac{\epsilon \theta_\epsilon}{N} +\frac{\psi_\eta^*}{N-n} \right)^2
 +\trace(\sigma_\epsilon^2)+ \Ric_N(\eta^*)=0,
\end{equation}
and for $N=+\infty$, $\theta_\epsilon$ satisfies
\begin{equation} \label{fsq}
\theta_\epsilon^* +(1-\epsilon^2) \frac{\theta_\epsilon^2}{n}
 +\frac{1}{n} (\epsilon \theta_\epsilon +\psi_\eta^*)^2
 +\trace(\sigma_\epsilon^2)+ \Ric_\infty(\eta^*) =0.
\end{equation}
\end{theorem}

\begin{proof}
The first equation \eqref{jsp} is obtained as the trace of \eqref{eq:w-Ricc} by noticing
\[ \trace(\sB_\epsilon^2) =\trace\left( \sigma_\epsilon^2
 +\frac{2\theta_\epsilon}{n} \sigma_\epsilon +\frac{\theta_\epsilon^2}{n^2}I_n \right)
 =\trace(\sigma_\epsilon^2) +\frac{\theta_\epsilon^2}{n}. \]
Then \eqref{fsp} follows from \eqref{jsp} by comparing $\Ric_0$ and $\Ric_N$.
The expression \eqref{fsq} for $N=+\infty$ can be derived again from \eqref{jsp},
or as the limiting case of \eqref{fsp}.
$\qedd$
\end{proof}

The usefulness of \eqref{fsp} and \eqref{fsq} stands in the possibility of controlling the positivity
of the coefficient in front of $\theta_\epsilon^2$, as we shall see.
Though we did not have a Raychaudhuri equation with this property for  $N=n$,
we do have a meaningful \emph{Raychaudhuri inequality}.

\begin{proposition}[Timelike weighted Raychaudhuri inequality]\label{pr:Ray-ineq} %$\empty$ \\
Let $\sJ$ be a nonsingular Lagrange tensor field along a timelike geodesic
$\eta:I \lra M$ of unit speed.
For every $\epsilon \in \R$ and $N\in (-\infty,0) \cup [n,+\infty]$, we have on $I$
\begin{equation} \label{dos}
\theta_{\epsilon}^*
 \le -{\Ric_N}(\eta^*) -\trace(\sigma_{\epsilon}^2) -c \theta_{\epsilon}^2,
\end{equation}
where
\begin{equation}\label{eq:cNe}
c=c(N,\epsilon):=\frac{1}{n}\left(1-\epsilon^2\frac{N-n}{N}\right).
\end{equation}
Moreover, for $\epsilon=0$ one can take $N \to 0$
and \eqref{dos} holds with $c=c(0,0):=1/n$.
\end{proposition}

\begin{proof}
For $N\in (-\infty,0) \cup (n,+\infty]$, the inequality \eqref{dos} readily follows from
\eqref{fsp} or \eqref{fsq}.
The case of $N=n$ is obtained by taking the limit $N \downarrow n$.
The case of $N=\epsilon=0$ is immediate from \eqref{jsp}.
$\qedd$
\end{proof}

Looking at the condition for $c>0$, we arrive at a key step for singularity theorems.

\begin{proposition}[Timelike $\epsilon$-range for convergence] \label{bdo} %$\empty$\\
Given $N \in (-\infty,0] \cup [n,+\infty]$, take $\epsilon\in \R$ such that
\begin{equation} \label{ran}
\epsilon=0 \text{ for } N=0, \quad
 \vert\epsilon\vert < \sqrt{\frac{N}{N-n}} \text{ for } N\neq 0.
\end{equation}
Let $\eta:(a,b) \lra M$ be a timelike geodesic of unit speed.
Assume that $\Ric_N({\eta}^*) \ge 0$ holds on $(a,b)$,
and let $\sJ$ be a Lagrange tensor field along $\eta$ such that for some $t_0 \in (a,b)$
we have $\theta_\epsilon(t_0) <0$.
Then we have $\det\sJ(t)=0$ for some $t \in [t_0,t_0 +s_0]$ provided that $t_0 +s_0<b$,
where we set, with $c=c(N,\epsilon)>0$ in \eqref{eq:cNe},
\begin{equation}\label{eq:s_0}
s_0:=\tau_\epsilon^{-1} \bigg( \tau_\epsilon(t_0) -\frac{1}{c\theta_\epsilon(t_0)} \bigg) -t_0.
\end{equation}

Similarly, if $\theta_\epsilon(t_0)>0$,
then we have $\det\sJ(t)=0$ for some $t \in [t_0+s_0,t_0]$
provided that  $t_0+s_0 >a$ for $s_0$ above.
\end{proposition}

Note that the assumption $\Ric_N(\eta^*) \ge 0$ is equivalent to $\Ric_N(\dot{\eta}) \ge 0$,
and that $\theta_\epsilon(t_0)<0$ is equivalent to $\theta_\psi(t_0)<0$
(corresponding to $\epsilon=0$).
When $N=n$, the condition \eqref{ran} is void and we can take any $\epsilon \in \R$.

\begin{proof}
Let us consider the former case of $\theta_\epsilon(t_0)<0$, then $s_0>0$.
Observe that $\theta_\epsilon(t_0)^{-1} =c(\tau_\epsilon(t_0)-\tau_\epsilon(t_0+s_0))$.
Assume to the contrary that $[t_0,t_0 +s_0] \subset (a,b)$ and
$\det\sJ(t) \neq 0$ for all $t \in [t_0,t_0 +s_0]$.
Since $\sigma_\epsilon$ is $g_{\dot \eta}$-symmetric,
we deduce from \eqref{dos} that $\theta_{\epsilon}^* \le -c{\theta_{\epsilon}^2} \le 0$.
Hence we have $\theta_\epsilon<0$ on $[t_0,b)$ and,
moreover, $[\theta_{\epsilon}^{-1}]^* \ge c$.
Integrating this inequality from $t_0$ to $t \in (t_0,t_0+s_0)$ yields
\[ \theta_{\epsilon}(t)
 \le \frac{1}{\theta_{\epsilon}(t_0)^{-1} +c(\tau_\epsilon(t)-\tau_\epsilon(t_0))}
 =\frac{1}{c(\tau_\epsilon(t)-\tau_\epsilon(t_0+s_0))}<0. \]
This implies $\lim_{t \uparrow t_0+s_0}\theta_{\epsilon}(t) =-\infty$.
Then, since
\[ \theta_{\epsilon}
 =\e^{\frac{2(1-\epsilon)}{n}\psi_\eta} \trace (\sB) -\psi_{\eta}^*
 =\e^{\frac{2(1-\epsilon)}{n}\psi_\eta} \frac{(\det \sJ)'}{\det\sJ} -\psi_{\eta}^*, \]
it necessarily holds that $\det\sJ(t_0+s_0)=0$, a contradiction.
The case of $\theta_{\epsilon}(t_0)>0$ (where $s_0<0$) is proved analogously.
$\qedd$
\end{proof}

\begin{remark}[Admissible range of $\epsilon$]\label{rm:eps}
The condition \eqref{ran} for $\epsilon$ gives an important insight on the relation
between $N$ and the admissible range of $\epsilon$.
Observe that $\epsilon=0$ as in \cite{WW1,WW2} is allowed for any $N \in (-\infty,0] \cup [n,+\infty]$,
while $\epsilon=1$ corresponding to the usual proper time is allowed only for $N \in [n,+\infty)$.
\end{remark}

\subsection{Completenesses}\label{ssc:cplt}%%%%%%%%%%%%
%%%%%%%%%%%%%%%%%%%

Inspired by Proposition~\ref{bdo},
we introduce a completeness condition associated with the $\epsilon$-proper time in \eqref{par}.

\begin{definition}[Timelike $\epsilon$-completeness]\label{df:e-cplt}
Let $\eta:(a,b) \lra M$ be an inextendible timelike geodesic.
We say that $\eta$ is \emph{future $\epsilon$-complete}
 if $\lim_{t \to b}\tau_\epsilon(t)=+\infty$.
Similarly, we say that it is \emph{past $\epsilon$-complete} if $\lim_{t \to a}\tau_\epsilon(t)=-\infty$.
The spacetime $(M,L,\psi)$ is said to be \emph{future timelike $\epsilon$-complete}
if every inextendible timelike geodesic is future $\epsilon$-complete,
and similar in the past case.
\end{definition}

If $\epsilon=1$ one simply speaks of the (geodesic) completeness
with respect to the usual proper time (namely $b=+\infty$),
while if $\epsilon=0$ one speaks of the \emph{$\psi$-completeness}
introduced by Wylie \cite{Wy} in the Riemannian case
and by Woolgar--Wylie \cite{WW1,WW2} in the Lorentzian case.
Note also that the $\epsilon$-completeness was tacitly assumed in \cite{Ca,GW}
through the upper boundedness of $\psi$ (see Lemma~\ref{lm:cplt} below).
The following corollary is immediate from Proposition~\ref{bdo}.

\begin{corollary} \label{mci}
Let $N \in (-\infty,0]\cup [n,+\infty]$ and $\sJ$ be a Lagrange tensor field
along a future inextendible timelike geodesic $\eta : (a,b) \lra M$
satisfying $\Ric_N(\dot\eta) \ge 0$.
Assume that $\eta$ is future $\epsilon$-complete for some $\epsilon \in \R$
that belongs to the timelike $\epsilon$-range in \eqref{ran},
and that $\theta_\epsilon(t_0)<0$ for some $t_0 \in (a,b)$.
Then $\eta$ develops a point $t \in (t_0,b)$ where $\det\sJ(t)=0$.
\end{corollary}

\begin{proof}
It suffices to show that one can always find $s_0 \in (0,b- t_0)$ satisfying
$\theta_\epsilon(t_0)^{-1} =c(\tau_\epsilon(t_0)-\tau_\epsilon(t_0+s_0))$.
This clearly holds true under the future $\epsilon$-completeness.
$\qedd$
\end{proof}

We remark that the future $\epsilon$-completeness clearly requires the future inextendability,
but not necessarily $b=+\infty$.
The next lemma is an immediate consequence of Definition~\ref{df:e-cplt},
see \cite[Lemma~1.3]{WW1}.

\begin{lemma}\label{lm:cplt}
Let $\epsilon<1$.
If $\psi$ is bounded above, then the future $($resp.\ past$)$ completeness implies
the future $($resp.\ past$)$ $\epsilon$-completeness.
If $\psi_\eta$ is non-increasing along every timelike geodesic $\eta$,
then the future completeness implies the future $\epsilon$-completeness.
Similarly, if $\psi_\eta$ is non-decreasing along every timelike geodesic $\eta$,
then the past completeness implies the past $\epsilon$-completeness.
\end{lemma}

\subsection{Timelike geodesic congruence from a point}%%%%%%%%%%%%
%%%%%%%%%%%%%%%%%%%

In this subsection, we study timelike geodesic congruences issued from a point.
A similar analysis can be applied to timelike geodesic congruences
that are  orthogonal to a spacelike hypersurface.
Our objective is to show that they determine Lagrange tensor fields.
This subsection does not use the weight.

\begin{proposition} \label{cong}
Let $(M,L)$ be a Finsler spacetime, and let $\eta:[0,l] \lra M$ be a timelike geodesic of unit speed.
Suppose that there is no point conjugate to $\eta(0)$ along $\eta$.
Then there exists a Lagrange tensor field $\sJ(t):N_{\eta}(t) \lra N_{\eta}(t)$
such that $\sJ(0)=0$, $\sJ'(0)=I_n$ and $\det \sJ(t)>0$ for all $t \in (0,l]$.
\end{proposition}

\begin{proof}
Let $x=\eta(0)$ and $v=\dot \eta(0)$.
For each $w \in T_xM$, we consider the vector field $Y_w:=d(\exp_x)_{tv}(tw) \in T_{\eta(t)}M$.
By construction it is a Jacobi field along $\eta$ satisfying $Y_w(0)=0$ and $Y'_w(0)=w$,
where we denote by $Y'_w$ the covariant derivative
$D^{\dot{\eta}}_{\dot{\eta}}Y_w$ along $\eta$.

We shall define an endomorphism $\sJ(t):N_{\eta}(t) \lra N_{\eta}(t)$
(in a way similar to Remark~\ref{rm:Jtensor}).
Given $w \in N_{\eta}(t)$, we extend it to the $g_{\dot{\eta}}$-parallel vector field $P$ along $\eta$
(namely $P' \equiv 0$), and then define $\sJ(t)(w):=Y_{P(0)}(t)$.
Note that the image of $\sJ(t)$ is indeed included in $N_{\eta}(t)$, since it follows from
\eqref{eq:g_eta}, \eqref{eq:Jacobi} and Proposition~\ref{pr:R_v} that
\[ \frac{d^2}{dt^2} \big[ g_{\dot{\eta}}(\dot{\eta}, Y_{P(0)}) \big]
 =-g_{\dot{\eta}}\big( \dot{\eta}, R_{\dot{\eta}}(Y_{P(0)}) \big) \equiv 0. \]

Since $P$ is $g_{\dot{\eta}}$-parallel, we have
\[ \sJ'(P) =\sJ(P)' -\sJ(P') =Y'_{P(0)}, \qquad
 \sJ''(P) =(\sJ'(P))' -\sJ'(P') =Y''_{P(0)} =-R_{\dot{\eta}}(Y_{P(0)}). \]
Therefore $\sJ$ satisfies the equation $\sJ''+\sR \sJ=0$.
Since $\eta(0)$ has no conjugate point by hypothesis and $Y_{P(0)}(0)=0$,
the map $\sJ(t)$ has maximum rank and hence invertible for every $t \in (0,l]$.
In particular, $\ker(\sJ(t)) \cap \ker(\sJ'(t)) =\{0\}$ for all $t \in [0,l]$, thus $\sJ$ is a Jacobi tensor field.
%(recall Definition~\ref{df:Jtensor}).

Next, we prove that $\sJ^{\sT} \sJ'$
%that the endomorphism $\sB(t):N_{\eta}(t) \lra N_{\eta}(t)$ defined by
%\begin{equation}\label{eq:B}
%\sB:=\sJ'\sJ^{-1}
%\end{equation}
is $g_{\dot \eta}$-symmetric.
%The previously shown fact, that $Y'_{P(0)}$ is $g_{\dot \eta}$-orthogonal to $\dot \eta$,
%proves indeed that $\sB(t)$ has image in $N_\eta(t)$.
%We observed that this property is equivalent to the symmetry of $\sJ^{\sT} \sB \sJ=\sJ^{\sT} \sJ'$,
%where the transposition $\sT$ is based on $g_{\dot \eta}$.
To this end, observe that
\[ \frac{d}{dt}\big[ g_{\dot{\eta}}(Y'_{w_1},Y_{w_2}) -g_{\dot{\eta}}(Y_{w_1},Y'_{w_2}) \big]
 =-g_{\dot{\eta}}\big( R_{\dot{\eta}}(Y_{w_1}),Y_{w_2} \big)
 +g_{\dot{\eta}}\big( Y_{w_1},R_{\dot{\eta}}(Y_{w_2}) \big) =0 \]
for $w_1,w_2 \in N_{\eta}(0)$, where we used \eqref{eq:symm}.
Combining this with $Y_{w_1}(0)=Y_{w_2}(0)=0$ yields
%\[
$g_{\dot{\eta}}(Y'_{w_1},Y_{w_2})=g_{\dot{\eta}}(Y_{w_1},Y'_{w_2})$.
%\]
This shows that $\sJ^{\sT} \sJ'$ is indeed symmetric (and hence $\sJ$ is a Lagrange tensor field)
because, for the $g_{\dot{\eta}}$-parallel vector field $P_i$ with $P_i(0)=w_i$ ($i=1,2$),
\[
g_{\dot{\eta}} \big( P_1,\sJ^{\sT} \sJ'(P_2) \big)
 =g_{\dot{\eta}}(Y_{w_1},Y'_{w_2})=g_{\dot{\eta}}(Y'_{w_1},Y_{w_2})
 =g_{\dot{\eta}} \big( \sJ^{\sT} \sJ' (P_1), P_2 \big).
\]

Finally, we find by construction that $\sJ(0)=0$ and $\sJ'(0)=I_n$,
where $I_n$ is the identity of $N_\eta(0)$.
Thus we obtain, for $t$ sufficiently close to $0$, $\det \sJ(t)=\det(tI_n +o(t))>0$.
By the continuity and non-degeneracy of $\sJ$, $\det\sJ(t)$ is indeed positive for every $t$.
$\qedd$
\end{proof}

\subsection{Comparison theorems}\label{ssc:w-comp}%%%%%%%%
%%%%%%%%%%%%%%%%%%%%%%%

This subsection is devoted to the (weighted) Lorentz--Finsler analogues of
two fundamental comparison theorems in Riemannian geometry,
the \emph{Bonnet--Myers} and \emph{Cartan--Hadamard theorems}.
We refer to \cite{Ch} for the Riemannian case, \cite{BCS} for the Finsler case,
and to \cite[Chapter~11]{BEE} for the Lorentzian case.

Though we will give precise proofs, it is also possible to reduce those theorems
to the (weighted) Lorentzian setting by using Theorem~\ref{th:curv}.
We refer to \cite{Onlga} for details.

\begin{proposition}[Weighted Bishop inequality] \label{th:wBin} %$\empty$ \\
Let $\sJ$ be a nonsingular Lagrange tensor field along a timelike geodesic
$\eta:I \lra M$ of unit speed.
Let $N\in (-\infty,0] \cup [n,+\infty]$ and
$\epsilon\in \R$ be in the timelike $\epsilon$-range as in \eqref{ran}.
Defining $\xi:=|{\det \sJ_\psi}|^{c}$ with $c>0$ in \eqref{eq:cNe}, we have on $I$
\[ \xi^{**}\le - c \xi \Ric_N(\eta^*). \]
\end{proposition}

\begin{proof}
Note that $\sJ$ being nonsingular ensures that $\det \sJ_\psi$ is always positive or always negative.
If $\det \sJ_\psi>0$, then we deduce from $\log \xi =c \log(\det \sJ_\psi)$ that
\[ \frac{\xi^*}{\xi} =c \frac{(\det \sJ_\psi)^*}{\det \sJ_\psi}
 =c \trace(\sJ^*_\psi \sJ_\psi^{-1})=c \trace(\sB_\epsilon)=c \theta_\epsilon. \]
Thus $\xi^{**}\xi-(\xi^*)^2=c\theta_\epsilon^* \xi^2$,
and then the weighted Raychaudhuri inequality \eqref{dos} yields
\[ \xi^{**} \le -c\xi \{ \Ric_N(\eta^*) +\trace(\sigma_\epsilon^2)\} \le -c\xi \Ric_N(\eta^*). \]
In the case of $\det \sJ_\psi<0$, we have $\log \xi =c \log({-\det \sJ_\psi})$
and can argue similarly.
$\qedd$
\end{proof}

An interesting case is $N\in [n,+\infty)$, $\epsilon=1$ and $c=1/N$,
for it corresponds to the usual proper time parametrization
and leads us to the weighted Bonnet--Myers theorem.

We are going to need some auxiliary geometric properties of Finsler spacetimes.
The existence of convex neighborhoods implies that several standard proofs from causality theory,
originally developed for Lorentzian spacetimes, pass unaltered to the Lorentz--Finsler framework
(we refer to \cite{Min-conv}).
An important result is a generalization of the \emph{Avez--Seifert connectedness theorem} as follows
(see \cite[Proposition~6.9]{Min-Ray},
it actually holds under much weaker regularity assumptions on the metric
as in \cite[Theorem~2.55]{Min-causality}).

\begin{theorem}[Avez--Seifert theorem]\label{th:connect}
In a globally hyperbolic Finsler spacetime,
any two causally related points are connected by a maximizing causal geodesic.
\end{theorem}

%In particular, in globally hyperbolic Finsler spacetimes, the Lorentz--Finsler distance is finite.
It should be recalled here that in a Finsler spacetime two points connected by a causal curve
which is not a lightlike geodesic are necessarily connected by a timelike curve,
see \cite[Lemma~2]{Min-conv} or \cite[Theorem~2.16]{Min-causality}.
Thus a lightlike curve which is maximizing is necessarily a lightlike geodesic.

We also need the following
(see \cite[Proposition~5.1]{Min-Ray} and also \cite[Theorem~6.16]{Min-Rev}).
%Recall Definition~\ref{df:conj} for the definition of conjugate points.

\begin{proposition}[Beyond conjugate points]\label{bjs}
In a Finsler spacetime, a causal geodesic $\eta:[a,b] \lra M$ cannot be maximizing
if it contains an internal point conjugate to $\eta(a)$.
Similarly, a causal geodesic $\eta:(a,b) \lra M$ cannot be maximizing
if it contains a pair of mutually conjugate points.
\end{proposition}

Define the \emph{timelike diameter} of a Finsler spacetime $(M,L)$ by
\[ \diam(M) :=\sup\{ d(x,y) \,|\, x,y \in M \}. \]
By the definition of the distance function,
given $x,y \in M$ and any causal curve $\eta$ from $x$ to $y$,
we have $\ell(\eta) \le d(x,y)$.
Hence, if $\diam(M)<\infty$, then every timelike geodesic has finite length
and $(M,L)$ is timelike geodesically incomplete (see \cite[Remark~11.2]{BEE}).

Now we state a weighted Lorentz--Finsler analogue of the Bonnet--Myers theorem.

\begin{theorem}[Weighted Bonnet--Myers theorem]\label{th:wBoMy} %$\empty$\\
Let $(M,L,\psi)$ be globally hyperbolic of dimension $n+1 \ge 2$.
If $\Ric_N \ge K$ holds in timelike directions for some $N \in [n,+\infty)$ and $K>0$,
then we have
\[ \diam(M) \le \pi\sqrt{\frac{N}{K}}. \]
\end{theorem}

\begin{proof}
Suppose that the claim is not true, then we can find two causally related points $x,y \in M$
such that $d(x,y)>\pi\sqrt{N/{K}}$.
By Theorem~\ref{th:connect}, there is a timelike geodesic $\eta:[0,l] \lra M$ with
$\eta(0)=x$, $\eta(l)=y$, $F(\dot \eta)=1$ and $l=\ell(\eta)=d(x,y)>\pi\sqrt{N/K}$.
We are going to prove that, due to $l>\pi\sqrt{N/{K}}$,
there necessarily exists a conjugate point to $\eta(0)$.
Then Proposition~\ref{bjs} gives the desired contradiction.

Now we assume that there is no conjugate point to $\eta(0)$.
Then Proposition~\ref{cong} applies and we have a Lagrange tensor field $\sJ$
with the properties given there.
Define $\sJ_\psi=\e^{-\psi_\eta/n} \sJ$ and $\xi=({\det \sJ_\psi})^{1/N}$ (i.e., $\epsilon=1$),
and notice that $\xi >0$ for $t>0$.
Then by Proposition~\ref{th:wBin} with $\epsilon=1$, we have 
\begin{equation} \label{gor}
N \xi''(t)\le- \xi(t) {\Ric_N} \big( \dot\eta(t) \big) \le -K\xi(t). 
\end{equation}
Putting $\bs(t):=\sin(t\sqrt{K/N})$, we obtain $(\xi' \bs -\xi \bs')' \le 0$.

Let us prove that $\lim_{t \to 0} (\xi' \bs -\xi \bs')(t) \le 0$,
from which it follows $\xi' \bs -\xi \bs' \le 0$.
Notice that $\xi\in C^2((0,l])\cap C^0([0,l])$ and $\xi(0)=0$ (by $\sJ(0)=0$),
so we need only to prove $\lim_{t \to 0} \xi'(t) t \le 0$ where $\xi'(t) t$ needs not be $C^1$ at $0$.
We deduce from \eqref{gor} that $\xi$ is concave in $t$ near $t=0$.
Let $f(t): =  \xi(t)-t \xi'(t)$ be the ordinate of the intersection
between the tangent to the graph of $\xi$ at $(t,  \xi(t))$ and the vertical axis.
By the concavity of $\xi$, $f$ is non-decreasing in $t$ and $f(t) \ge \xi(0)=0$.
Therefore the limit $\lim_{t \to 0}f(t)$ exists and we obtain
\[
 \lim_{t \to 0} t  \xi'(t) =-\lim_{t \to 0} f(t) \le 0.
\]

Since $\xi' \bs -\xi \bs' \le 0$, the ratio $\xi(t)/\bs(t)$ is non-increasing in $t \in (0,\pi\sqrt{N/K})$.
Hence $\xi(t_0)=0$ necessarily holds at some $t_0 \in (0,\pi\sqrt{N/K}]$.
This contradicts the assumed absence of conjugate points,
therefore we conclude that $\diam(M) \le \pi\sqrt{N/K}$.
$\qedd$
\end{proof}

We remark that the unweighted situation is included in the above theorem as $\psi=0$
and then we have $\diam(M) \le \pi\sqrt{n/K} =\pi\sqrt{(\dim M -1)/K}$.

We end this section by giving a Lorentz--Finsler version of the Cartan--Hadamard theorem,
that we obtain in the unweighted case only.

\begin{theorem}[Cartan--Hadamard theorem]\label{th:CaHa} %$\empty$ \\
Let $(M,L)$ be a globally hyperbolic Finsler spacetime
whose flag curvature $\bK(v,w)$ is nonpositive for every $v \in \Omega_x$
and linearly independent $w \in T_xM$.
%and assume that the flag curvature $\bK(v,w)$ is nonpositive
%for every $x \in M$, $v \in \Omega_x$, and $w \in T_xM$ linearly independent from $v$.
Then every causal geodesic does not have  conjugate points.
\end{theorem}

We remark that our flag curvature has the opposite sign to \cite{BEE},
thus we are considering the nonpositive curvature (similarly to the Riemannian or Finsler case).

\begin{proof}
Assume that there is a timelike geodesic $\eta:[0,l] \lra M$
and a nontrivial Jacobi field $Y$ along $\eta$ such that $Y(0)$ and $Y(l)$ vanish.
We will denote by $Y'$ the covariant derivative $D^{\dot{\eta}}_{\dot{\eta}}Y$ along $\eta$.
We deduce from
\[ \frac{d^2}{dt^2} \big[ g_{\dot{\eta}}(\dot{\eta},Y) \big]
 =-g_{\dot{\eta}} \big( \dot{\eta},R_{\dot{\eta}}(Y) \big) =0 \]
that $g_{\dot{\eta}}(\dot{\eta}(t),Y(t))$ is affine in $t$, but it vanishes at $t=0,l$.
This implies that $g_{\dot{\eta}}(\dot{\eta},Y) \equiv 0$ and $g_{\dot{\eta}}(\dot{\eta},Y') \equiv 0$.
Hence $Y$ and $Y'$ are $g_{\dot{\eta}}$-spacelike and, in particular,
$g_{\dot{\eta}}(Y,Y) \ge 0$ as well as $g_{\dot{\eta}}(Y',Y') \ge 0$.
The assumption implies $g_{v}\big( w,R_{v}(w) \big) \le 0$ for $v\in \Omega_x$ and $w\in T_xM$
(which by continuity implies the same inequality for $v\in \overline{\Omega}_x$).
Thus we have
\[ \frac{d^2}{dt^2} \big[ g_{\dot{\eta}}(Y,Y) \big]
 =2g_{\dot{\eta}}(Y',Y') -2g_{\dot{\eta}}\big( Y,R_{\dot{\eta}}(Y) \big)
 \ge 0. \]
Therefore $g_{\dot{\eta}}(Y,Y)$ is a nonnegative convex function vanishing at $t=0,l$,
and hence $g_{\dot{\eta}}(Y,Y)=0$.
This implies that $Y$ vanishes on entire $[0,l]$, a contradiction.

For a lightlike geodesic $\eta$, we obtain $g_{\dot{\eta}}(Y,Y)=0$ by the same argument,
and then $Y \equiv 0$ if $Y(t)$ is $g_{\dot{\eta}}$-spacelike at some $t \in (0,l)$.
In the case where $Y$ is always $g_{\dot{\eta}}$-lightlike,
since $g_{\dot{\eta}}(\dot{\eta},Y)=0$,
we have $Y(t)=f(t) \dot{\eta}(t)$ for some function $f$ with $f(0)=f(l)=0$.
Combining this with $f''=0$ following from the Jacobi equation of $Y$, we have $f \equiv 0$.
$\qedd$
\end{proof}

In the Riemannian or Finsler setting,
the absence of conjugate points yields that the exponential map $\exp_x:T_xM \lra M$
is a covering and, if $M$ is simply-connected, $\exp_x$ is a diffeomorphism.
The Lorentzian case is not as simple as such
since Theorem~\ref{th:CaHa} is concerned with only causal geodesics.
See \cite[Section~11.3]{BEE} for further discussions.

\section{Null case}\label{sc:n-Ray}%%%%%%
%%%%%%%%%%%%%%%%%%%%

The arguments in Subsections~\ref{ssc:w-Ricc}--\ref{ssc:cplt}
can be extended to lightlike geodesics.
We will keep the same notations $\tau_\epsilon$ and $c$ for quantities
that are just analogous to those appearing in the timelike case
(compare \eqref{pae} and \eqref{eq:n-cNe} in this section with \eqref{par} and \eqref{eq:cNe},
respectively), hoping that this choice will cause no confusion.

Let $\eta: I \lra M$ be a future-directed lightlike geodesic,
i.e., $L(\dot\eta)=0$ and $\dot\eta \neq 0$.
Then $N_{\eta}(t) \subset T_{\eta(t)}M$ is similarly defined as
the $n$-dimensional subspace $g_{\dot{\eta}(t)}$-orthogonal to $\dot{\eta}(t)$,
but in this case $\dot{\eta}(t) \in N_{\eta}(t)$.
Thus it is convenient to work with the quotient space
\[ Q_\eta(t) :=N_{\eta}(t)/ \dot \eta(t). \]
The metric $g_{\dot \eta}$ induces the positive-definite metric $h$ on this quotient bundle over $\eta$.
It can be shown (see \cite{Min-Ray}) that the covariant derivative $D^{\dot\eta}_{\dot\eta}$
is well defined over this quotient, and it can be extended linearly
over the space of endomorphisms of $Q_\eta(t)$.
It is important to observe that this vector space is $(n-1)$-dimensional,
so its identity $I_{n-1}$ has a trace which equals $n-1$.
This fact explains why in passing from the timelike to the null case
we get the replacements $n \mapsto n-1$ and $N \mapsto N-1$ in several formulas.
Jacobi and Lagrange tensor fields are endomorphisms of this space
but are otherwise defined in the usual way (Definition~\ref{df:Jtensor}).
For instance, a Jacobi tensor field $\sJ$ satisfies
\begin{equation} \label{gsf}
\sJ''+\sR \sJ=0
\end{equation}
(where $'$ is the mentioned covariant derivative on the quotient space)
and $\ker(\sJ(t)) \cap \ker(\sJ'(t)) =\{0\}$
(this 0 belongs to $Q_\eta(t)$, if we work with endomorphisms of $N_{\eta}(t)$
then we would have $\R \dot\eta(t)$ on the right-hand side).
In \eqref{gsf}, $\sR: Q_\eta \lra Q_\eta$ is the $h$-symmetric curvature endomorphism.
Then $\sB:=\sJ'\sJ^{-1}$ is also an $h$-symmetric endomorphism of $ Q_\eta$,
and $\sigma$ and $\theta$ are its trace and traceless parts (similarly to Definition~\ref{df:expan}),
see \cite{Min-Ray} for details.

Analogous to the $\epsilon$-proper time \eqref{par} in the timelike case,
along a lightlike geodesic $\eta$ we define
\begin{equation}\label{pae}
\tau_\epsilon:=\int \e^{\frac{2(\epsilon-1)}{n-1}\psi_\eta(t)} \,dt.
\end{equation}
Similarly to the previous section, we denote by $*$ the (covariant) derivative in $\tau_\epsilon$,
thus $\eta^*(t)=\e^{\frac{2(1-\epsilon)}{n-1}\psi_\eta(t)} \dot\eta(t)$.
The \emph{weighted Jacobi endomorphism}
\[ \sJ_\psi(t) := \e^{-\psi_\eta(t)/(n-1)} \sJ(t) \]
and the \emph{curvature endomorphism}
\[ \sR_{(N,\epsilon)}(t) := \e^{\frac{4(1-\epsilon)}{n-1} \psi_\eta(t)}
 \left\{ \sR(t)+\frac{1}{n-1}\bigg(\psi_\eta''(t)-\frac{\psi'_{\eta}(t)^2}{N-n}\bigg) I_{n-1} \right\} \]
are defined in the same way as well.
Notice that $\trace(\sR_{(N,\epsilon)})=\Ric_N(\eta^*)$.
The same calculation as Lemma~\ref{lm:w-Jacobi} yields the \emph{weighted Jacobi equation}
%\begin{equation}\label{eq:nw-Jacobi}
\[ \sJ_\psi^{**} +\frac{2\epsilon}{n-1} \psi_\eta^* \sJ_\psi^*+ \sR_{(1,\epsilon)} \sJ_\psi=0, \]
%\end{equation}
where we remark that
$\sR_{(1,\epsilon)}$ is employed instead of $\sR_{(0,\epsilon)}$ in \eqref{eq:w-Jacobi}.

For $t \in I$ where $\sJ(t)$ is invertible, we define
\[ \sB_{\epsilon} :=\sJ_\psi^* \sJ_\psi^{-1}
 =\e^{\frac{2(1-\epsilon)}{n-1}\psi_\eta} \left(\sB -\frac{\psi'_{\eta}}{n-1} I_{n-1}\right). \]
Then the \emph{weighted Riccati equation}
\begin{equation}\label{eq:nw-Ricc}
\sB_\epsilon^* +\frac{2\epsilon}{n-1} \psi_\eta^* \sB_\epsilon+\sB_\epsilon^2 +\sR_{(1,\epsilon)}=0
\end{equation}
is obtained similarly to Lemma~\ref{lm:w-Ricca}.
We also define the \emph{$\epsilon$-expansion scalar}
\[ \theta_{\epsilon}(t) :={\trace} \big( \sB_{\epsilon}(t) \big)
 =\e^{\frac{2(1-\epsilon)}{n-1}\psi_\eta(t)} \big(\theta(t) -\psi'_{\eta}(t)\big) \]
and the \emph{$\epsilon$-shear tensor}
\[ \sigma_\epsilon(t) :=\sB_{\epsilon}(t) -\frac{\theta_{\epsilon}(t)}{n-1} I_{n-1}
 = \e^{\frac{2(1-\epsilon)}{n-1}\psi_\eta(t)} \sigma(t). \]
Taking the trace of the weighted Riccati equation \eqref{eq:nw-Ricc},
we get the \emph{weighted Raychaudhuri equation}, in the same manner as Theorem~\ref{th:wRay}.

\begin{theorem}[Null weighted Raychaudhuri equation]\label{th:wnRay} %$\empty$\\
Let $\sJ$ be a nonsingular Lagrange tensor field along a future-directed lightlike geodesic
$\eta:I \lra M$.
Then, for $N=1$, the $\epsilon$-expansion $\theta_\epsilon$ satisfies
\[ \theta_\epsilon^* +\frac{2 \epsilon}{n-1} \psi_\eta^* \theta_\epsilon+\frac{\theta_\epsilon^2}{n-1}
 +\trace(\sigma_\epsilon^2)+ \Ric_1(\eta^*)=0 \]
on $I$.
For $N \in (-\infty,+\infty)\backslash \{1,n\}$, it satisfies
\begin{align*}
\theta_\epsilon^* +\left( 1-\epsilon^2 \frac{N-n}{N-1} \right) \frac{\theta_\epsilon^2}{n-1}
 +\frac{(N-1)(N-n)}{n-1} \left( \frac{\epsilon \theta_\epsilon}{N-1} +\frac{\psi_\eta^*}{N-n} \right)^2&
 \\
 +\trace(\sigma_\epsilon^2) +\Ric_N(\eta^*)
 &=0,
\end{align*}
and for $N=+\infty$ it satisfies
\[ \theta_\epsilon^*+(1-\epsilon^2)\frac{\theta_\epsilon^2}{n-1}
 +\frac{1}{n-1} (\epsilon \theta_\epsilon +\psi_\eta^*)^2
 +\trace(\sigma_\epsilon^2)+ \Ric_\infty(\eta^*) =0. \]
\end{theorem}

Once again the usefulness of these equations stands in the possibility
of controlling the positivity of the coefficient in front of $\theta_\epsilon^2$.
The analogues to Propositions~\ref{pr:Ray-ineq} and \ref{bdo} hold as follows.

\begin{proposition}[Null weighted Raychaudhuri inequality]\label{pr:nRay-ineq} %$\empty$ \\
Let $\sJ$ be a nonsingular Lagrange tensor field along a lightlike geodesic
$\eta:I \lra M$.
For every $\epsilon \in \R$ and $N\in (-\infty,1) \cup [n,+\infty]$, we have on $I$
\begin{equation} \label{dog}
\theta_{\epsilon}^*
 \le -{\Ric_N}(\eta^*) -\trace(\sigma_{\epsilon}^2) -c \theta_{\epsilon}^2,
\end{equation}
where
\begin{equation}\label{eq:n-cNe}
c =c(N,\epsilon) =\frac{1}{n-1}\left(1-\epsilon^2\frac{N-n}{N-1}\right).
\end{equation}
Moreover, for $\epsilon=0$ one can take $N \to 1$ and \eqref{dog} holds with $c=c(1,0)=1/(n-1)$.
\end{proposition}

\begin{proposition}[Null $\epsilon$-range for convergence] \label{bdp} %$\empty$\\
Given $N \in (-\infty,1] \cup [n,+\infty]$, take $\epsilon \in \R$ such that
\begin{equation} \label{ram}
\epsilon=0 \text{ for } N=1, \quad \vert \epsilon\vert < \sqrt{\frac{N-1}{N-n}} \text{ for } N \ne 1.
\end{equation}
Let $\eta:(a,b) \lra M$ be a lightlike geodesic.
Assume that $\Ric_N({\eta}^*) \ge 0$ holds on $(a,b)$,
and let $\sJ$ be a Lagrange tensor field along $\eta$ such that
for some $t_0 \in (a,b)$ we have  $\theta_\epsilon(t_0)<0$.
Then we have $\det\sJ(t)=0$ for some $t \in [t_0,t_0+s_0]$
provided that $t_0+s_0 <b$,
where $c$ and $s_0$ are from \eqref{eq:n-cNe} and \eqref{eq:s_0}, respectively.

Similarly, if $\theta_\epsilon(t_0)>0$,
then we have $\det\sJ(t)=0$ for some $t \in [t_0+s_0,t_0]$ provided that $t_0+s_0 >a$.
\end{proposition}

Similarly to Remark~\ref{rm:eps},
note that in \eqref{ram} $\epsilon=0$ is allowed for any $N$,
while $\epsilon=1$ is allowed only for $N \in [n,+\infty)$.
We proceed to the study of completeness conditions.

\begin{definition}[Null $\epsilon$-completeness]\label{df:ne-cplt}
Let $\eta:(a,b) \lra M$ be an inextendible lightlike geodesic.
We say that $\eta$ is \emph{future $\epsilon$-complete} (resp.\ \emph{past $\epsilon$-complete})
if $\lim_{t \to b}\tau_\epsilon(t)=+\infty$ (resp.\ $\lim_{t \to a}\tau_\epsilon(t)=-\infty$).
The spacetime $(M,L,\psi)$ is said to be \emph{future null $\epsilon$-complete}
if every lightlike geodesic is future $\epsilon$-complete, and similar in the past case.
\end{definition}

The next corollary is obtained similarly to Corollary~\ref{mci}.

\begin{corollary} \label{bhd}
Let $N \in (-\infty,1] \cup [n,+\infty]$ and $\sJ$ be a Lagrange tensor field
along a future inextendible lightlike geodesic $\eta:(a,b) \lra M$
satisfying $\Ric_N(\dot\eta) \ge 0$.
Assume that $\eta$ is future $\epsilon$-complete for some $\epsilon \in \R$
that satisfies \eqref{ram}, and that $\theta_\epsilon(t_0)<0$ for some $t_0 \in (a,b)$.
Then $\eta$ develops a point $t \in (t_0,b)$ where $\det \sJ(t)=0$.
\end{corollary}

\section{Incomplete or conjugate}\label{sc:conj}%%%%%%%%%%
%%%%%%%%%%%%%%%%%%%%%%

In this section we show that, under some genericity and convergence conditions,
every timelike or lightlike geodesic is either incomplete or including a pair of conjugate points.
The following notion will play an essential role.

\begin{definition}[Genericity conditions]\label{df:t-gene}
Let $\eta:(a,b) \lra M$ be a timelike geodesic of unit speed.
We say that the \emph{genericity condition} holds along $\eta$
if there exists $t_1 \in (a,b)$ such that $\sR(t_1) \neq 0$,
where  $\sR(t)=R_{\dot{\eta}(t)}:N_{\eta}(t) \lra N_{\eta}(t)$.
We say that $(M,L,\psi)$ satisfies the \emph{timelike genericity condition}
if the genericity condition holds along every inextendible timelike geodesic.
Similarly, we define the \emph{null genericity condition}
where this time we use the curvature endomorphism on the quotient space $Q_\eta$.
We say that $(M,L,\psi)$ satisfies the \emph{causal genericity condition}
if it satisfies both the timelike and null genericity conditions.
\end{definition}

\begin{remark}\label{rm:t-gene}
This is the standard genericity condition for Lorentz--Finsler geometry (see \cite{Min-Ray})
which generalizes that of Lorentzian geometry (see for instance \cite{BEE}).

In the timelike case, we need to introduce a weighted version only in the extremal case $N=0$,
where we replace $\sR$ with $\sR_{(0,0)}$ from \eqref{mfi} similarly to \cite{Ca,WW2},
see Remarks~\ref{rm:N=0}, \ref{ngp} for further discussions.
Also for $N \neq 0$, we could use the weighted version in the next results,
Lemma~\ref{lm:Beem12.14} and Proposition~\ref{pr:Beem12.10},
with no alteration in the conclusions.
This is because in the relevant step of the proof one observes that $\psi_\eta'=0$
and hence all the curvature endomorphisms coincide up to a multiplicative factor.

In the null case, we need a weighted version only in the extremal case $N=1$,
where we replace $\sR$ with $\sR_{(1,0)}$.
Again for $N \neq 1$, we could use the weighted version in the next results
with no alteration in the conclusions.
\end{remark}

\begin{definition}\label{df:L+-}
Let $\eta:(a,b) \lra M$ be an inextendible timelike geodesic of unit speed.
For $t \in (a,b)$, define $L_+(t)$ (resp.\ $L_-(t)$) as the collection of all Lagrange tensor fields
$\sJ$ along $\eta$ such that $\sJ(t)=I_{n}$ and $\theta_1(t) \ge 0$ (resp.\ $\theta_1(t) \le 0$).
\end{definition}

Recall from \eqref{eq:w-expan} that $\theta_1=\theta-\psi'_\eta$
and that $\theta_1(t) \ge 0$ is equivalent to $\theta_\epsilon(t) \ge 0$
regardless of the choice of $\epsilon$.

\begin{lemma}\label{lm:Beem12.14}
Let $N\in (-\infty,0) \cup [n,+\infty]$
and $\eta:(a,b) \lra M$ be an inextendible timelike geodesic of unit speed such that
$\Ric_N(\dot{\eta}) \ge 0$ on $(a,b)$ and $\sR(t_1)\neq 0$ for some $t_1\in\R$.
\begin{enumerate}[{\rm (i)}]
\item \label{it:conj-1}
Suppose that $\eta$ is future $\epsilon$-complete
where $\epsilon\in \R$ belongs to the timelike $\epsilon$-range in \eqref{ran}.
Then, for any $\sJ \in L_-(t_1)$, there exists some $t \in (t_1,b)$ such that $\det\sJ(t)=0$.

\item \label{it:conj-2}
Similarly, if $\eta$ is past $\epsilon$-complete for $\epsilon$ in \eqref{ran},
then for any $\sJ \in L_+(t_1)$ there exists some $t \in (a,t_1)$ such that $\det\sJ(t)=0$.
\end{enumerate}
\end{lemma}

\begin{proof}
Since the proofs are similar, we prove only \eqref{it:conj-1}.
The condition $\sJ \in L_-(t_1)$ implies $\theta_\epsilon(t_1) \le 0$.
If there is some $t_0 \ge t_1$ such that $\theta_\epsilon(t_0)<0$,
then Corollary \ref{mci} shows the existence of $t>t_1$ with $\det\sJ(t)=0$.
Thus we assume $\theta_\epsilon(t) \ge 0$ for all $t \ge t_1$.

It follows from the Raychaudhuri inequality \eqref{dos} that $\theta'_{\epsilon}(t) \le 0$,
hence $\theta_{\epsilon}(t)=0$ for all $t \ge t_1$.
Then the Raychaudhuri equation \eqref{fsp} or \eqref{fsq} implies that $\Ric_N(\dot{\eta}(t))=0$,
$\trace(\sigma_\epsilon(t)^2) =0$ and $\psi'_\eta(t)=0$ for all $t \ge t_1$.
(For the case $N=n$, we take $N' \in (n,\infty)$
such that $\epsilon$ belongs to the timelike $\epsilon$-range
of $N'$ and apply \eqref{fsp} for $N'$ with the help of $\Ric_{N'} \ge \Ric_n$ from \eqref{eq:mono}).
Since $\sigma_\epsilon$ is $g_{\dot \eta}$-symmetric, we have $\sigma_\epsilon(t)=0$ for all $t \ge t_1$.
Moreover, we deduce from \eqref{eq:w-expan}, \eqref{eq:w-shear} and \eqref{eq:w-B}
that $\theta(t)=0$, $\sigma(t)=0$ and $\sB(t)=0$ for all $t \ge t_1$.
Then we obtain from the unweighted Riccati equation $\sB'+\sB^2+\sR=0$ in \eqref{eq:Ricca}
that $\sR(t)=0$ for all $t \ge t_1$, a contradiction that completes the proof.
$\qedd$
\end{proof}

\begin{remark}[$N=0$ case]\label{rm:N=0}
In the extremal case of $N=0$ (and hence $\epsilon=0$),
the same argument as Lemma~\ref{lm:Beem12.14} shows $\theta_{\epsilon}(t)=0$
and it implies $\Ric_0(\dot{\eta}(t))=0$,
$\sigma_\epsilon(t)=0$ and $\sB_\epsilon(t)=0$, but not $\psi'_\eta(t)=0$ (see \eqref{jsp}).
Nonetheless, the weighted Riccati equation \eqref{eq:w-Ricc} yields $\sR_{(0,0)}(t)=0$,
therefore we obtain the same conclusion as Lemma~\ref{lm:Beem12.14}
by replacing the hypothesis $\sR(t_1) \neq 0$ with the weighted genericity condition
$\sR_{(0,0)}(t_1) \neq 0$ similar to \cite{Ca,WW2}.
This phenomenon could be compared with Wylie's observation in the splitting theorems:
One obtains the isometric splitting for $N \in (-\infty,0) \cup [n,+\infty]$,
while for $N=0$ only the weaker \emph{warped product splitting} holds true.
We refer to \cite{Wy} for the Riemannian case and \cite{WW2} for the Lorentzian case
(where $N=1$ is the extremal case due to the difference from our notation,
recall Remark~\ref{rm:n+1}).
\end{remark}

The following proposition is the next key step towards singularity theorems.

\begin{proposition}[Generating conjugate points]\label{pr:Beem12.10}
Let $N \in (-\infty,0) \cup [n,+\infty]$ and $\epsilon \in \R$
belong to the timelike $\epsilon$-range in \eqref{ran}.
Let $\eta: (a,b) \lra M$ be an $\epsilon$-complete timelike geodesic
satisfying the genericity condition and $\Ric_N(\dot{\eta}) \ge 0$ on $(a,b)$.
Then $\eta$ necessarily has a pair of conjugate points.
\end{proposition}

To prove the proposition,
we need two lemmas on Lagrange tensor fields shown in the same way
as the Lorentzian setting.
Indeed, everything can be calculated in terms of $g_{\dot{\eta}}$,
thereby one can follow the same lines as \cite[Lemmas~12.12, 12.13]{BEE}.

\begin{lemma}\label{lm:Beem12.12}
Let $\eta:(a,b) \lra M$ be a timelike geodesic of unit speed having no conjugate points.
Take $t_1 \in (a,b)$ and let $\sJ$ be the unique Lagrange tensor field along $\eta$
such that $\sJ(t_1)=0$ and $\sJ'(t_1)=I_{n}$.
Then, for each $s \in (t_1,b)$,
the Lagrange tensor field $\sD_s$ with $\sD_s(t_1)=I_{n}$ and $\sD_s(s)=0$
satisfies the equation
\begin{equation}\label{eq:D_s}
\sD_s(t) =\sJ(t)\int_t^s (\sJ^{\sT} \sJ)(r)^{-1} \,dr
\end{equation}
for all $t \in (t_1,b)$.
Moreover, $\sD_s(t)$ is nonsingular for all $t \in (t_1,s)$.
\end{lemma}

\begin{proof}
Recall that $\sJ'$ means $D_{\dot\eta}^{\dot\eta} \sJ$.
Note first that by the standard ODE theory
the Jacobi tensor field $\sJ$ is uniquely determined by the boundary condition
$\sJ(t_1)=0$ and $\sJ'(t_1)=I_n$.
Moreover, $\sJ(t_1)=0$ ensures that $\sJ$ is a Lagrange tensor field
(recall Remark~\ref{rm:Jtensor}).

The endomorphism in the right-hand side of \eqref{eq:D_s},
\[ \mathsf{X}(t) :=\sJ(t)\int_t^s (\sJ^{\sT} \sJ)(r)^{-1} \,dr, \qquad t \in (t_1,b), \]
is well defined since $\eta$ has no conjugate points and $\sJ(t_1)=0$.
We shall see that $\mathsf{X}$ is a Lagrange tensor field satisfying the same boundary condition
as $\sD_s$ at $s$, which implies $\sD_s=\mathsf{X}$.
The condition $\mathsf{X}''+\sR \mathsf{X} =0$ for $\mathsf{X}$ being a Jacobi tensor field
is proved by using the symmetry \eqref{eq:Ltensor} for $\sJ$.
Since $\mathsf{X}(s)=0$ clearly holds, $\mathsf{X}$ is indeed a Lagrange tensor field.
Moreover, we deduce from $[(\sJ^{\sT})' \sD_s -\sJ^{\sT} \sD'_s]' \equiv 0$
(by the symmetry \eqref{eq:symm}), $\sJ(t_1)=0$ and $\sJ'(t_1)=\sD_s(t_1)=I_{n}$ that
$(\sJ^{\sT})' \sD_s -\sJ^{\sT} \sD'_s \equiv I_{n}$.
Hence
\[ \mathsf{X}'(s) =-\sJ(s) \cdot (\sJ^{\sT} \sJ)(s)^{-1}
 =-\sJ^{\sT}(s)^{-1} =\sD'_s(s). \]
Therefore we obtain $\sD_s=\mathsf{X}$.
The nonsingularity for $t \in (t_1,s)$ is seen by noticing that $(\sJ^{\sT} \sJ)(r)^{-1}$ is positive-definite.
$\qedd$
\end{proof}

\begin{lemma}\label{lm:Beem12.13}
Let $\eta:(a,b) \lra M$ be a timelike geodesic of unit speed without conjugate points.
For $t_1 \in (a,b)$ and $s \in (t_1,b)$,
let $\sJ$ and $\sD_s$ be the Lagrange tensor fields as in Lemma~$\ref{lm:Beem12.12}$.
Then $\sD(t):=\lim_{s \to b} \sD_s(t)$ exists and is a Lagrange tensor field along $\eta$
such that $\sD(t_1)=I_n$ and $\sD'(t_1)=\lim_{s \to b} \sD'_s(t_1)$.
Moreover, $\sD(t)$ is nonsingular for all $t \in (t_1,b)$.
\end{lemma}

\begin{proof}
We can argue along the lines of \cite[Lemma 12.13]{BEE}
(by replacing $a$ in that proof with any $a' \in (a, t_1)$ in our notation)
and find that $\lim_{s \to b} \sD'_s(t_1)$ exists and $\sD$ is the Lagrange tensor field
such that $\sD(t_1)=I_{n}$ and $\sD'(t_1)=\lim_{s \to b} \sD'_s(t_1)$,
represented as
\[ \sD(t)=\sJ(t)\int_t^{b} (\sJ^{\sT} \sJ)(r)^{-1} \,dr, \qquad t \in (t_1,b). \]
The nonsingularity is shown in the same way as Lemma~\ref{lm:Beem12.12}.
$\qedd$
\end{proof}

We are ready to prove Proposition~\ref{pr:Beem12.10}.
Notice that we will use both \eqref{it:conj-1} and \eqref{it:conj-2} of Lemma~\ref{lm:Beem12.14},
so that both the future and past $\epsilon$-completenesses are required.

\begin{prooof}
Suppose that $\eta$ has no conjugate points and fix $t_1 \in (a,b)$ such that $\sR(t_1) \neq 0$.
Let $\sD:=\lim_{s \to b} \sD_s$ be the Lagrange tensor field given in Lemma~\ref{lm:Beem12.13},
and $\theta_1(t)$ be the $1$-expansion associated to $\sD$.
Thanks to Lemma~\ref{lm:Beem12.14}\eqref{it:conj-1} and the nonsingularity of $\sD$ on $(t_1,b)$,
we have $\sD \not\in L_-(t_1)$ and hence $\theta_1(t_1)>0$.
Since $\sD(t_1)=\lim_{s \to b} \sD_s(t_1)$ and $\sD'(t_1)=\lim_{s \to b} \sD'_s(t_1)$,
$\theta_1(t_1)>0$ still holds for $\sD_s$ with sufficiently large $s>t_1$.
Then it follows from Lemma~\ref{lm:Beem12.14}\eqref{it:conj-2} that there exists $t_2<t_1$
such that $\det\sD_s(t_2)=0$.

Now, take $v \in N_{\eta}(t_2) \setminus \{0\}$ with $\sD_s(t_2)(v)=0$ and
let $P$ be the $g_{\dot{\eta}}$-parallel vector field along $\eta$ with $P(t_2)=v$.
Then, $Y:=\sD_s(P)$ is a Jacobi field (recall Remark~\ref{rm:Jtensor}) and we have
\[ Y(t_2)=\sD_s(t_2)(v)=0, \qquad Y(s)=0, \qquad Y(t_1)=P(t_1) \neq 0. \]
Therefore $\eta(s)$ is conjugate to $\eta(t_2)$, a contradiction.
This completes the proof.
$\qedd$
\end{prooof}

An analogous proof gives the following result for null geodesics.

\begin{proposition}\label{pr:Beem12.17}
Let $N \in (-\infty,1) \cup [n,+\infty]$ and
$\epsilon \in \mathbb{R}$ belong to the null $\epsilon$-range in \eqref{ram}.
Let $\eta: (a,b) \lra M$ be an $\epsilon$-complete lightlike geodesic
satisfying the genericity condition and $\Ric_N(\dot{\eta}) \ge 0$ on $(a,b)$.
Then $\eta$ necessarily has a pair of conjugate points.
\end{proposition}

We summarize the outcomes of Propositions~\ref{pr:Beem12.10} and \ref{pr:Beem12.17}
by using the following notion.

\begin{definition}[Convergence conditions]\label{df:t-conv}
We say that $(M,L,\psi)$ satisfies the \emph{timelike $N$-convergence condition}
(resp.\ the \emph{null $N$-convergence condition})
for $N \in (-\infty,+\infty] $ if we have $\Ric_N(v) \ge 0$ for all timelike vectors $v \in \Omega$
(resp.\ for all $v \in \partial {\Omega}$).
\end{definition}

By  continuity, the timelike $N$-convergence condition
is equivalent to $\Ric_N(v) \ge 0$ for all causal vectors $v \in \overline{\Omega}$,
so it can also be called the \emph{causal $N$-convergence condition}.

\begin{theorem}\label{th:Beem12.18}
Let $(M,L,\psi)$ be a Finsler spacetime of dimension $n +1 \ge 2$,
satisfying the timelike genericity and timelike $N$-convergence conditions
for some $N \in (-\infty,0) \cup [n,+\infty]$.
Then every future-directed timelike geodesic is either including a pair of conjugate points
or $\epsilon$-incomplete for any $\epsilon \in \R$ belonging to the timelike $\epsilon$-range \eqref{ran}.
\end{theorem}

By the \emph{$\epsilon$-incompleteness}, we mean that (at least)
one of the future and past $\epsilon$-completenesses fails.
In the null case we have similarly the next result.

\begin{theorem}\label{vjx}
Let $(M,L,\psi)$ be a Finsler spacetime of dimension $n +1 \ge 3$,
satisfying the null genericity and null $N$-convergence conditions
for some $N \in (-\infty,1) \cup [n,+\infty]$.
Then every future-directed lightlike geodesic is either including a pair of conjugate points
or $\epsilon$-incomplete for any $\epsilon \in \R$ belonging to the null $\epsilon$-range \eqref{ram}.
\end{theorem}

\begin{remark}[Extremal cases]\label{ngp}
Due to Remark~\ref{rm:N=0},
when $N=0$ in the timelike case or $N=1$ in the null case,
we have the analogues to Theorems~\ref{th:Beem12.18}, \ref{vjx}
under the modified genericity conditions $\sR_{(0,0)}(t_1) \neq 0$
or $\sR_{(1,0)}(t_1) \neq 0$ at some $t_1 \in (a,b)$.
\end{remark}

\section{Singularity theorems}\label{sc:sing}%%%%%%%%%%%%%
%%%%%%%%%%%%%%%%%%%%%%%%%%%%%%%%%%

We finally discuss several singularity theorems derived from the results in the previous sections
(recall Subsection~\ref{com} for the general strategy).
Our presentation follows \cite{Min-Ray} based on causality theory (see also \cite{AJ}).
We also refer to \cite{Min-causality,Min-Rev} for singularity theorems in causality theory.
Recall Subsection~\ref{ssc:causal} for some notations in causality theory.

\subsection{Trapped surfaces}\label{ssc:trap}%%%%%%%%%%%%%
%%%%%%%%%%%%%%%%%%%%%%%%%%%%%%%%%%

We first introduce the notion of trapped surfaces.
Let $S \subset M$ be a co-dimension $2$, orientable, compact $C^2$-spacelike submanifold
without boundary.
By this we mean that for each $x \in S$, $T_x S \cap \overline{\Omega}_x =\{0\}$.
By the convexity of the cone $\overline{\Omega}_x$
there are exactly two hyperplanes $H_x^{\pm} \subset T_xM$ containing $T_xS$
and tangent to $\overline{\Omega}_x$.
These hyperplanes determine two future-directed lightlike vectors $v^{\pm}$
in the sense that $H_x^\pm$ intersects $\overline{\Omega}_x$
in the ray $\R_+ v^{\pm}$.
This fact can be seen as a consequence of the bijectivity of the Legendre map,
and we have $H_x^{\pm}=\ker g_{v^\pm}(v^\pm, \cdot)$
(see \cite[Proposition~3]{Min-cone} and also \cite[Proposition~5.2]{AJ}).
A $C^1$-choice of the vector field $v^\pm$ over $S$ will be denoted by $V^\pm$.
It exists by the orientability provided that the spacetime is orientable in a neighborhood of $S$,
and is uniquely determined up to a point-wise rescaling, $V^\pm \mapsto f V^\pm$, with $f>0$.

Now we consider the \emph{geodesic congruence} generated by $V^+$,
namely the family of lightlike geodesics emanating from $S$ with the initial condition $V^+$.
Let $\eta: [0,b) \lra M$, with $x:=\eta(0) \in S$ and $\dot\eta(0)=V^+(x)$, be one such geodesic.
Then we consider the Jacobi tensor field $\sJ$ along $\eta$ associated with the geodesic congruence,
namely $\sJ(0)=I_{n-1}$ and $\sJ'(0)(w)= D^{V^+}_w V^+$ for each $w \in Q_\eta(0)$.
(We remark that this is an endomorphism left unchanged by the above rescaling
(thereby well defined), i.e., invariant under the replacements $w \mapsto w+f V^+(x)$.
Hence it is enough to consider $w \in T_xS$).
More intuitively, given $w \in T_xS$ and the $g_{\dot\eta}$-parallel vector field $P$ with $P(0)=w$,
the Jacobi field $Y_w:=\sJ(P)$ satisfies $Y_w(0)=w$ and $Y'_w(0)=D^{V^+}_w V^+$
so that $Y_w$ is the variational vector field of a geodesic variation
$\zeta:[0,b) \times (-\ve,\ve) \lra M$ such that $\zeta(0,\cdot)$ is a curve in $S$ with
$\del_s\zeta(0,0)=w$ and $\zeta(\cdot,s)$ is the geodesic with initial vector $V^+(\zeta(0,s))$
for each $s$.

One can show that $\sJ$ is in fact a Lagrange tensor field
(see \cite[Section~4]{Min-Ray}, and this could be compared with the symmetry
of the Hessian in the positive-definite case as in \cite[Lemma~2.3]{OSbw}).
That is, let $w_1,w_2\in T_xS$ and extend them to two vector fields $W_1,W_2$ tangent to $S$ and
commuting at $x$ (i.e., $[W_1,W_2](x)=0$).
Next extend them to a neighborhood $U$ of $S$ with no focal points.
Let us also extend $V^+$ to a vector field on $U$,
and let us keep the same notations for the extended fields.
Since $W_i$ is tangent to $S$, we have $\del_{w_i} g_{V^+}(V^+, W_j)=0$ for $i,j=1,2$.
Then it follows from \eqref{eq:g_X} and $[W_1,W_2](x)=0$ that
\[
g_{V^+}(w_1, D^{V^+}_{w_2} V^+) =-g_{V^+}(D^{V^+}_{w_2}W_1,  V^+)
 =-g_{V^+}(D^{V^+}_{w_1}W_2,  V^+) =g_{V^+}(w_2, D^{V^+}_{w_1} V^+).
\]
This together with $\sJ(0)=I_{n-1}$ implies the symmetry of $\sB=\sJ'\sJ^{-1}$,
and hence the Lagrange property for $\sJ$ (recall Remark~\ref{rm:Jtensor}).

\emph{Focal points} of $S$ are those at which $\det \sJ=0$.
In Lorentz--Finsler geometry it has been proved in \cite[Proposition~5.1]{Min-Ray}
(see \cite[Theorem~6.16]{Min-Rev} for the analogous Lorentzian proof)
that every geodesic of the congruence including a focal point
necessarily enters the set $I^+(S)$ defined in Subsection~\ref{ssc:causal}
(this result does not use the weight and so passes to our case).
A future \emph{lightlike $S$-ray} is a future inextendible, lightlike geodesic $\eta: [0,b) \lra M$
such that $\eta(0) \in S$ and $d(S,\eta(t))=\ell(\eta\vert_{[0,t]})$ for all $t \in (0,b)$.
Then $\eta$ issues necessarily orthogonally from $S$, and does not intersect $I^+(S)$.
Note also that, if every geodesic of the congruence develops a focal point,
then there are no future lightlike $S$-rays.

The \emph{expansion} $\theta^+:S \lra \R$ of $S$ will be the expansion
of the geodesic congruence defined by
\[ \theta^+(x) :=\trace(\sJ' \sJ^{-1})(0)=\trace(w \mapsto D^{V^+}_w V^+), \]
where $\sJ$ is the Lagrange tensor field along the geodesic $\eta$
with $\dot{\eta}(0)=V^+(x)$ as above.
The right-hand side can be interpreted as the trace of the shape operator of $S$.
Similarly, define the \emph{$\epsilon$-expansion} $\theta^+_\epsilon :S \lra \R$ of $S$ by
\[ \theta_\epsilon^+(x) :=\trace(\sJ_\psi^* \sJ_\psi^{-1})(0)
 =\e^{\frac{2(1-\epsilon)}{n-1}\psi_\eta(0)} \big( \theta^+(x) -\psi'_{\eta}(0) \big). \]
The factor on the right-hand side is in most cases of no importance,
since what really matters is the sign of $\theta_\epsilon^+$.
For instance the constant $\epsilon$ does not appear in the next definition.
We define $\theta^-$ and $\theta^-_\epsilon$ associated with $V^-$ in the same manner.

\begin{definition}[Trapped surfaces]\label{jbp}
We say that $S$ is a \emph{$\psi$-trapped surface}
if $\theta_1^+ <0$ and $\theta_1^- <0$ on $S$.
\end{definition}

By the null Raychaudhuri equation of Theorem~\ref{th:wnRay},
more precisely by Corollary~\ref{bhd}, we obtain the following.

\begin{proposition}\label{vip}
Let $(M,L,\psi)$ be a Finsler spacetime of dimension $n +1 \ge 3$,
satisfying the null  $N$-convergence condition
for some $N \in (-\infty,1] \cup [n,+\infty]$.
Let $S$ be a $\psi$-trapped surface.
Then every lightlike $S$-ray is necessarily future $\epsilon$-incomplete
for any $\epsilon \in \R$ that belongs to the null $\epsilon$-range \eqref{ram}.
\end{proposition}

\begin{proof}
Assume to the contrary that a lightlike $S$-ray is future $\epsilon$-complete
for some $\epsilon$ satisfying \eqref{ram}.
By Corollary~\ref{bhd} it develops a focal point,
hence by \cite[Proposition~5.1]{Min-Ray} it enters $I^+(S)$,
which contradicts the definition of a future lightlike $S$-ray.
$\qedd$
\end{proof}

\subsection{Singularity theorems}\label{ssc:sing}%%%%%%%%%%%%%
%%%%%%%%%%%%%%%%%%%%%%%%%%%%%%%%%%

In the previous sections we generalized Step~\ref{step1} according to the general strategy
outlined in Subsection~\ref{com}.
Thus, we are ready to obtain some notable singularity theorems
in the weighted Lorentz--Finsler framework.

%For the next step we comment on the typical structure of singularity theorems
%(see \cite[Section~6.6]{Min-Rev} for further discussions).
%They are composed of the following three steps:
%
%\begin{enumerate}[I.]
%\item\label{step1}
%A non-causal statement assuming some form of geodesic completeness
%plus some genericity and convergence conditions,
%and implying the existence of conjugate points in geodesics or focal points
%for certain (hyper)surfaces with special convergence properties,
%e.g., our Corollaries~\ref{mci} and \ref{bhd}.
%This step typically makes use of the Raychaudhuri equation.
%
%\item\label{step2}
%A non-causal statement to the effect that the presence of conjugate or focal points
%spoils some length maximization property (achronal property in the null case),
%for instance \cite[Proposition~5.1]{Min-Ray} used to show Proposition~\ref{vip}.
%
%\item\label{step3}
%A statement to the effect that under some causality conditions
%as well as in presence of some special set (trapped set, Cauchy hypersurface)
%the spacetime necessarily has a \emph{causal line} (a maximizing inextendible causal geodesic)
%or a causal $S$-ray.
%\end{enumerate}
%
%The first two results go in contradiction with the last one,
%so from here one infers the geodesic incompleteness.
We say that $S \subset M$ is \emph{achronal} if $I^+(S) \cap S=\emptyset$
(namely, no two points in $S$ are connected by a timelike curve).
A nonempty set $S \subset M$ is called a \emph{future trapped set}
if the \emph{future horismos} $E^+(S):=J^+(S) \setminus I^+(S)$ of $S$ is nonempty and compact.
Recall Definition~\ref{df:Cauchy} for the definition of Cauchy hypersurfaces.

%Interestingly, the first two steps basically coincide for all the singularity theorems.
%For instance, Penrose's and Gannon's singularity theorems \cite{Ga,P},
%but also the topological censorship theorem \cite{FSW},
%use the same versions of \ref{step1} and \ref{step2}.
%Similarly, Hawking and Penrose's and Borde's singularity theorems \cite{Bo,HP}
%use the same versions of \ref{step1} and \ref{step2}.
%Most singularity theorems really differ just for the causality statement in \ref{step3}.
%For this reason, it is often convenient to identify the singularity theorem
%with its \emph{causality core statement}, namely Step~\ref{step3}.
%It turns out that this causality core statement in most cases involves just the cone distribution,
%thereby it is fairly robust.

As mentioned in Subsection~\ref{com} we have the next causality core statement which is valid
for our Finsler spacetimes (and also for less regular spaces, \cite[Theorem~2.67]{Min-causality}).

\begin{theorem} \label{jvi}
Let $(M,L)$ be a Finsler spacetime admitting a non-compact Cauchy hypersurface.
Then every nonempty compact set $S$ admits a future lightlike $S$-ray.
\end{theorem}

Joining Proposition~\ref{vip} (as Steps~\ref{step1} and \ref{step2})
with Theorem~\ref{jvi} (as Step~\ref{step3}),
we obtain our first singularity theorem, which is a generalization of Penrose's theorem
(analogous to \cite[Theorem~6.25]{Min-Rev}).
Recall that a $\psi$-trapped surface is compact.

\begin{theorem}[Weighted Finsler Penrose's theorem]\label{th:Penrose} %$\empty$\\
Let $(M,L,\psi)$ be a Finsler spacetime of dimension $n +1 \ge 3$,
admitting a non-compact Cauchy hypersurface
and satisfying the null $N$-convergence condition for some $N \in (-\infty,1] \cup [n,+\infty]$.
Suppose that there is a $\psi$-trapped surface $S$.
Then there exists a lightlike geodesic issued from $S$ which is future $\epsilon$-incomplete
for every $\epsilon \in \R$ that belongs to the null $\epsilon$-range in \eqref{ram}.
\end{theorem}

As another example of causality core statement,
we consider the following theorem corresponding to
\cite[Theorem~2.64]{Min-causality} or \cite[Theorem~4.106]{Min-Rev}.
Recall that a \emph{time function} is a continuous function that increases over every causal curve.
For closed cone structures and hence for Finsler spacetimes,
the existence of a time function is equivalent to the \emph{stable causality},
i.e., the possibility of widening the causal cones without introducing closed causal curves
(see \cite[Theorem~2.30]{Min-causality}).
A \emph{lightlike line} is an inextendible lightlike geodesic for which
no two points can be connected by a timelike curve (i.e., achronality).

\begin{theorem} \label{bos}
Let $(M,L)$ be a chronological Finsler spacetime.
If there are no lightlike lines, then there exists a time function
and hence $(M,L)$ is stably causal.
\end{theorem}

Joining this with Theorem~\ref{vjx} and Proposition~\ref{bjs}
(as Steps~\ref{step1} and \ref{step2}, respectively),
we have a generalization of a singularity theorem obtained by the second author in \cite{Min-chr}.

\begin{theorem}[Absence of time implies singularities]\label{th:time} %$\empty$\\
Let $(M,L,\psi)$ be a chronological Finsler spacetime of dimension $n +1 \ge 3$,
satisfying the null genericity and the null $N$-convergence conditions
for some $N \in (-\infty,1) \cup [n,+\infty]$.
If there are no time functions, then there exists a lightlike line which is
$\epsilon$-incomplete for every $\epsilon \in \R$ belonging to the null $\epsilon$-range \eqref{ram}.
\end{theorem}

In the case of $N=1$, we have the same conclusion by replacing the genericity condition
with the weighted one $\sR_{(1,0)} \neq 0$ (recall Remarks~\ref{rm:N=0}, \ref{ngp}).

The next lemma from \cite[Corollary~2.117]{Min-Rev}
passes word-for-word to the Lorentz--Finsler case.
We say that $S \subset M$ is \emph{future null} (resp.\ \emph{causally}) \emph{araying}
if there are no future-directed lightlike (resp.\ causal) $S$-rays.

\begin{lemma} \label{ngk}
Let $(M,L)$ be a stably causal Finsler spacetime.
A nonempty compact set $S$ is a future trapped set if and only if it is future null araying.
\end{lemma}

Let us come to the causality core statement,
found in \cite[Theorem~2.71]{Min-causality} or \cite[Theorem~6.43]{Min-Rev},
behind Hawking and Penrose's theorem.

\begin{theorem} \label{th:min19}
Chronological Finsler spacetimes $(M,L)$ without causal lines
do not admit nonempty, compact, future null araying sets.
\end{theorem}

Notice that a chronological spacetime without lightlike lines is stably causal
by Theorem~\ref{bos} and hence, by Lemma~\ref{ngk},
future null araying sets in this statement can be equivalently replaced by future trapped sets.
Then the following is an analogue to \cite[Theorem~6.44]{Min-Rev}.
Given an achronal set $S \subset M$,
we define its \emph{edge} $\mathrm{edge}(S)$ as the set of points $x \in \overline{S}$
such that, for every neighborhood $U$ of $x$,
there are $y \in I^-(x;U) \setminus S$, $z \in I^+(x;U) \setminus S$
and a timelike curve in $U \setminus S$ from $y$ to $z$.
We denoted by $I^-(x;U)$ (resp.\ $I^+(x;U)$) the set of points $y \in U$
such that there is a smooth timelike curve in $U$
from $y$ to $x$ (resp.\ from $x$ to $y$).
An achronal set is a closed topological hypersurface
if and only if its edge is empty (see \cite[Corollary~14.26]{ON}).
%\cite[Definition~2.130]{Min-Rev}

\begin{theorem} [Weighted Finsler Hawking--Penrose's theorem]\label{th:HaPe} %$\empty$ \\
Let $(M,L,\psi)$ be a chronological Finsler spacetime
satisfying the causal genericity and the causal $N$-convergence conditions
for some $N \in (-\infty,0) \cup [n,+\infty]$.
Suppose that there exists one of the following:
\begin{enumerate}[{\rm (i)}]
\item
a compact achronal set without edge $($e.g., a compact achronal spacelike hypersurface$)$,
\item
a $\psi$-trapped surface,
\item
a point $x$ such that, on every lightlike geodesic emanating from $x$,
its expansion $\theta_1$ becomes negative at some point
$($i.e., the lightlike geodesic is \emph{reconverging}$)$.
\end{enumerate}
Then $(M,L,\psi)$ admits a timelike geodesic which is $\epsilon$-incomplete
for every $\epsilon \in \R$ belonging to the timelike $\epsilon$-range \eqref{ran},
or a lightlike geodesic which is $\epsilon$-incomplete for every $\epsilon \in \R$ satisfying \eqref{ram}.
In particular, it is $\psi$-incomplete
$($and incomplete in the usual sense if $N \in [n,\infty))$.
\end{theorem}

\begin{proof}
Suppose that the claim is not true.
Then, by Theorems~\ref{th:Beem12.18} and \ref{vjx},
every causal geodesic has conjugate points and hence is not maximizing, %\cite[Prop.\ 5.1]{Min-Ray},
thereby it is not a causal line.
A chronological spacetime without causal lines is stably causal (by Theorem~\ref{bos}),
thus compact future trapped sets and future null araying sets are the same (Lemma~\ref{ngk}).

(i)
A result of causality theory whose proof passes word-for-word to the Lorentz--Finsler case states that
every compact achronal set without edge is a future trapped set (see \cite[Corollary~2.145]{Min-Rev}),
hence a compact future null araying set.
This goes in contradiction with Theorem~\ref{th:min19}.

(ii)
Since a $\psi$-trapped surface is necessarily future null araying due to Proposition~\ref{vip}
and the hypothesis, this also goes in contradiction with Theorem~\ref{th:min19}.

(iii)
By Corollary~\ref{bhd} and \cite[Proposition~5.1]{Min-Ray},
every lightlike geodesic issued from $x$ enters $I^+(x)$,
namely the singleton $\{x\}$ is a compact future null araying set.
Therefore we have a contradiction again with Theorem~\ref{th:min19}.
$\qedd$
\end{proof}

\begin{remark}[$N=0$ case]\label{rm:N=0+}
A version for $N=0$ holds true, there we assume the standard null genericity condition
and the weighted timelike genericity condition demanding $\sR_{(0,0)} \neq 0$
in place of $\sR \neq 0$ at a point on each timelike geodesic (recall Remark~\ref{ngp}).
\end{remark}

We say that $S \subset M$ is \emph{acausal} if it does not admit $x,y \in S$ with $x<y$,
namely no causal curve meets $S$ more than once.
An acausal set is clearly achronal.
A \emph{partial Cauchy hypersuface} is by definition an acausal set without edge
(see \cite[Definition~3.35]{Min-Rev}).
The causal core statement which corresponds to Hawking's singularity theorem
is the following (see \cite[Theorem~6.48]{Min-Rev}).

\begin{theorem} \label{mom}
On a Finsler spacetime $(M,L)$ there is no compact partial Cauchy hypersurface $S$
which is future causally araying.
\end{theorem}

The concepts involved in this statement being dependent on the notion of Lorentz--Finsler length
are not purely causal.
Nevertheless, the proof uses only the existence of convex neighborhoods
and does indeed pass word-for-word to the Finsler setting.
We need a definition which is the analog of Definition~\ref{jbp} in the timelike case.

\begin{definition}[Contraction and expansion]\label{df:contr}
Let $S$ be a $C^2$-spacelike hypersurface, and $V$ be its future-directed normal vector field,
namely $V(x) \in \Omega_x$ and $\ker g_V(V(x),\cdot)= T_xS$ for all $x \in S$.
Consider the geodesic congruence generated by $V$,
the expansions $\theta=\trace(w \mapsto D^V_w V)$ and $\theta_\epsilon$ on $S$
in the same way as Subsection~\ref{ssc:trap}.
Then we say that $S$ is \emph{contracting} if $\theta<0$ on $S$,
and that $S$ is \emph{$\psi$-contracting} if $\theta_1<0$ on $S$.
If the inequality is reversed, then one speaks of \emph{expanding}
and \emph{$\psi$-expanding} hypersurfaces.
\end{definition}

\begin{theorem}[Weighted Finsler Hawking's theorem]\label{hah} %$\empty$ \\
Let $(M,L,\psi)$ be a Finsler spacetime satisfying the timelike $N$-convergence condition
for some $N \in (-\infty,0] \cup [n,+\infty]$.
If $M$ contains a compact $C^2$-spacelike hypersurface $S$ which is $\psi$-contracting,
then there exists a timelike geodesic issued normally from $S$ which is future $\epsilon$-incomplete
for every $\epsilon \in \R$ that belongs to the timelike $\epsilon$-range \eqref{ran}.
\end{theorem}

\begin{proof}
The proof goes as in \cite[Theorem~6.49]{Min-Rev}.
If $S$ is not acausal, then one can pass to the Geroch covering spacetime $M_G$
which contains an acausal homeomorphic copy of $S$ (see \cite[Section~2.15]{Min-Rev}).
Since the other assumptions lift to the covering spacetime,
and timelike geodesic $\epsilon$-incompleteness projects to the base,
we can assume that $S$ is acausal.
In particular, $S$ is achronal and a partial Cauchy hypersuface.

Assume that each timelike geodesic orthogonal to $S$ is future $\epsilon$-complete
for some $\epsilon \in \R$ that belongs to the timelike $\epsilon$-range in \eqref{ran}.
By Corollary~\ref{mci} and the hypothesis $\theta_1<0$,
every timelike geodesic issued normally from $S$ develops a focal point in the future,
thereby it cannot be a future causal $S$-ray (by \cite[Proposition~5.1]{Min-Ray}).
However, all future causal $S$-rays are necessarily orthogonal to $S$ and hence timelike,
therefore there are no future causal $S$-rays.
This shows that $S$ is future causally araying, a contradiction to Theorem~\ref{mom}.
$\qedd$
\end{proof}

\begin{remark}[Past case via reverse structure]\label{rm:reverse}
The past case of Theorem~\ref{hah}
can be seen by introducing the \emph{reverse structure} $\rev{L}(v):=L(-v)$.
Precisely, we consider the cone structure $\rev{\Omega}_x:=-\Omega_x$
and the weight $\rev{\psi}(v):=\psi(-v)$.
Then, for each timelike geodesic $\eta:(a,b) \lra M$ in $(M,L)$,
the reverse curve $\bar{\eta}(t):=\eta(-t)$ is a timelike geodesic in $(M,\rev{L})$,
and $\rev{\Ric}_N(\dot{\bar{\eta}}(t))=\rev{\Ric}_N(-\dot{\eta}(-t))=\Ric_N(\dot{\eta}(-t))$.
Now, assuming that $S$ is $\psi$-expanding with respect to $L$,
$S$ is $\rev{\psi}$-contracting with respect to $\rev{L}$ and Theorem~\ref{hah}
yields a timelike geodesic which is future $\epsilon$-incomplete for any $\epsilon$ in \eqref{ran}
with respect to $\rev{L}$.
Then its reverse curve is a timelike geodesic past $\epsilon$-incomplete
with respect to $L$, this gives the past case of Theorem~\ref{hah}.
\end{remark}
%\bigskip

{\it Acknowledgements}.
SO would like to thank Erasmo Caponio for drawing his attention to the subject of Finsler spacetimes.
SO was supported in part by JSPS Grant-in-Aid for Scientific Research (KAKENHI) 19H01786.
The authors are grateful to an anonymous referee for his/her valuable comments.

{\small%%%

}


\begin{thebibliography}{CGKM}%%%%%%%%%%%%%%%%%%%%

\bibitem[AJ]{AJ}
A.~B.~Aazami and M.~A.~Javaloyes,
Penrose's singularity theorem in a Finsler spacetime.
Classical Quantum Gravity {\bf 33} (2016), no.~2, 025003, 22 pp.

\bibitem[AB]{AB}
S.~B.~Alexander and R.~L.~Bishop,
Lorentz and semi-Riemannian spaces with Alexandrov curvature bounds.
Comm.\ Anal.\ Geom.\ {\bf 16} (2008), 251--282.

\bibitem[As]{As}
G.~S.~Asanov, Finsler geometry, relativity and gauge theories.
D.\ Reidel Publishing Co., Dordrecht, 1985.

\bibitem[BE]{BE}
D.~Bakry and M.~\'Emery, Diffusions hypercontractives. (French)
S\'eminaire de probabilit\'es, XIX, 1983/84, 177--206,
Lecture Notes in Math., {\bf 1123}, Springer, Berlin, 1985.

\bibitem[BGL]{BGL}
D.~Bakry, I.~Gentil and M.~Ledoux,
Analysis and geometry of Markov diffusion operators.
Springer, Cham, 2014.

\bibitem[BCS]{BCS}
D.~Bao, S.-S.~Chern and Z.~Shen, An introduction to Riemann--Finsler geometry.
Springer-Verlag, New York, 2000.

\bibitem[Be]{Be}
J.~K.~Beem, Indefinite Finsler spaces and timelike spaces.
Can.\ J.\ Math.\ {\bf 22} (1970), 1035--1039.

\bibitem[BEE]{BEE}
J.~K.~Beem, P.~E.~Ehrlich and K.~L.~Easley,
Global Lorentzian Geometry.
Marcel Dekker Inc., New York, 1996.

\bibitem[BS]{BS}
P.~Bernard and S.~Suhr,
Lyapounov functions of closed cone fields: From Conley theory to time functions.
Comm.\ Math.\ Phys.\ {\bf 359} (2018), 467--498.

\bibitem[BP]{BP}
J.~Bertrand and M.~Puel,
The optimal transport problem for relativistic costs.
Calc.\ Var.\ Partial Differential Equations {\bf 46} (2013), 353--374.

\bibitem[Bo]{Bo}
A.~Borde, Singularities in closed spacetimes.
Classical Quantum Gravity {\bf 2} (1985), 589--596.

\bibitem[Br]{Br}
Y.~Brenier, Extended Monge--Kantorovich Theory.
Optimal Transportation and Applications (Martina Franca, 2001).
Lecture Notes in Math., {\bf 1813}, 91--121. Springer, Berlin (2003).

\bibitem[Ca]{Ca}
J.~S.~Case,
Singularity theorems and the Lorentzian splitting theorem for the Bakry--Emery--Ricci tensor.
J.\ Geom.\ Phys.\ {\bf 60} (2010), 477--490.

\bibitem[Ch]{Ch}
I.~Chavel, Riemannian geometry. A modern introduction. Second edition.
Cambridge University Press, Cambridge, 2006.

\bibitem[CMS]{CMS}
D.~Cordero-Erausquin, R.~J.~McCann and M.~Schmuckenschl\"ager,
A Riemannian interpolation inequality \`a la Borell, Brascamp and Lieb.
Invent.\ Math.\ {\bf 146} (2001), 219--257.

\bibitem[CGKM]{CGKM}
P.~T.~Chru\'{s}ciel, J.~Grant, M.~Kunzinger and E.~Minguzzi,
Preface [Non-regular spacetime geometry].
J.\ Phys.\ Conf.\ Ser.\ {\bf 968} (2018), 011001, 3pp.

\bibitem[EM]{EM}
M.~Eckstein and T.~Miller,
Causality for nonlocal phenomena.
Ann.\ Henri Poincar\'e {\bf 18} (2017), 3049--3096.

%\bibitem[Es]{Es}
%J.-H.~Eschenburg,
%The splitting theorem for space-times with strong energy condition.
%J.\ Differential Geometry {\bf 27} (1988), 477--491.

\bibitem[FS]{FS}
A.~Fathi and A.~Siconolfi, On smooth time functions.
Math.\ Proc.\ Camb.\ Phil.\ Soc.\ {\bf 152} (2012), 303--339.

\bibitem[FSW]{FSW}
J.~L.~Friedman, K.~Schleich and D.~M.~Witt,
Topological censorship.
Phys.\ Rev.\ Lett.\ {\bf 71} (1993), 1486--1489.

%\bibitem[Ga]{Ga-split}
%G.~J.~Galloway,
%The Lorentzian splitting theorem without the completeness assumption.
%J.\ Differential Geometry {\bf 29} (1989), 373--387.

\bibitem[GW]{GW}
G.~J.~Galloway and E.~Woolgar,
Cosmological singularities in Bakry--\'Emery spacetimes.
J.\ Geom.\ Phys.\ {\bf 86} (2014), 359--369.

\bibitem[Ga]{Ga}
D.~Gannon, Singularities in nonsimply connected space-times.
J.\ Mathematical Phys.\ {\bf 16} (1975), 2364--2367.

\bibitem[GKS]{GKS}
J.~D.~E.~Grant, M.~Kunzinger and C.~S\"amann,
Inextendibility of spacetimes and Lorentzian length spaces.
Ann.\ Global Anal.\ Geom.\ {\bf 55} (2019), 133--147.

\bibitem[HE]{HE}
S.~W.~Hawking and G.~F.~R.~Ellis,
The large scale structure of space-time.
Cambridge University Press, London-New York, 1973.

\bibitem[HP]{HP}
S.~W.~Hawking and R.~Penrose,
The singularities of gravitational collapse and cosmology.
Proc.\ Roy.\ Soc.\ London Ser.\ A {\bf 314} (1970), 529--548.

\bibitem[KeSu]{KeSu}
M.~Kell and S.~Suhr,
On the existence of dual solutions for Lorentzian cost functions.
Ann.\ Inst.\ H.\ Poincar\'e Anal.\ Non Lin\'eaire {\bf 37} (2020), 343--372.

\bibitem[KM]{KM}
A.~V.~Kolesnikov and E.~Milman,
Brascamp--Lieb-type inequalities on weighted Riemannian manifolds with boundary.
J.\ Geom.\ Anal.\ {\bf 27} (2017), 1680--1702.

%\bibitem[Kr]{Kr}
%M. Kriele, Spacetime: foundations of general relativity and differential geometry.
%Springer-Verlag Berlin Heidelberg, 1999.

\bibitem[KuSa]{KuSa}
M.~Kunzinger and C.~S\"amann, Lorentzian length spaces.
Ann.\ Global Anal.\ Geom.\ {\bf 54} (2018), 399--447.

\bibitem[LV]{LV}
J.~Lott and C.~Villani,
Ricci curvature for metric-measure spaces via optimal transport.
Ann.\ of Math.\ {\bf 169} (2009), 903--991.

\bibitem[Mc]{Mc}
R.~J.~McCann,
Displacement convexity of Boltzmann's entropy characterizes the strong energy condition
from general relativity.
Camb.\ J.\ Math.\ {\bf 8} (2020), 609--681.

\bibitem[Mil]{Mil-neg}
E.~Milman,
Beyond traditional curvature-dimension I:
new model spaces for isoperimetric and concentration inequalities in negative dimension.
Trans.\ Amer.\ Math.\ Soc.\ {\bf 369} (2017), 3605--3637.

%\bibitem[Min1]{Min-lct}
%E.~Minguzzi, Limit curve theorems in Lorentzian geometry.
%J.\ Math.\ Phys.\ {\bf 49} (2008), 092501, 18pp.

\bibitem[Min1]{Min-chr}
E.~Minguzzi, Chronological spacetimes without lightlike lines are stably causal.
Comm.\ Math.\ Phys.\ {\bf 288} (2009), 801--819.

\bibitem[Min2]{Min-conv}
E.~Minguzzi, Convex neighborhoods for Lipschitz connections and sprays.
Monatsh.\ Math.\ {\bf 177} (2015), 569--625.

\bibitem[Min3]{Min-cone}
E.~Minguzzi, Light cones in Finsler spacetimes.
Comm.\ Math.\ Phys.\ {\bf 334} (2015), 1529--1551.

\bibitem[Min4]{Min-Ray}
E.~Minguzzi, Raychaudhuri equation and singularity theorems in Finsler spacetimes.
Classical Quantum Gravity {\bf 32} (2015), 185008, 26pp.

\bibitem[Min5]{Min-equiv}
E.~Minguzzi, An equivalence of Finslerian relativistic theories.
Rep.\ Math.\ Phys.\ {\bf 77} (2016), 45--55.

\bibitem[Min6]{Min-causality}
E.~Minguzzi, Causality theory for closed cone structures with applications.
Rev.\ Math.\ Phys.\ {\bf 31} (2019), 1930001, 139pp.

\bibitem[Min7]{Min-Rev}
E.~Minguzzi, Lorentzian causality theory.
Living Reviews in Relativity {\bf 22}, 3 (2019), https://doi.org/10.1007/s41114-019-0019-x.

\bibitem[MS]{MS}
A,~Mondino and S.~Suhr,
An optimal transport formulation of the Einstein equations of general relativity.
Preprint (2018). Available at {\sf arXiv:1810.13309}

\bibitem[Oh1]{Oint}
S.~Ohta, Finsler interpolation inequalities.
Calc.\ Var.\ Partial Differential Equations {\bf 36} (2009), 211--249.

\bibitem[Oh2]{ORand}
S.~Ohta, Vanishing S-curvature of Randers spaces.
Differential Geom.\ Appl.\ {\bf 29} (2011), 174--178.

\bibitem[Oh3]{Osplit}
S.~Ohta, Splitting theorems for Finsler manifolds of nonnegative Ricci curvature.
J.\ Reine Angew.\ Math.\ {\bf 700} (2015), 155--174.

\bibitem[Oh4]{Oneg}
S.~Ohta, $(K,N)$-convexity and the curvature-dimension condition for negative $N$.
J.\ Geom.\ Anal.\ {\bf 26} (2016), 2067--2096.

\bibitem[Oh5]{Onlga}
S.~Ohta, Nonlinear geometric analysis on Finsler manifolds.
Eur.\ J.\ Math.\ {\bf 3} (2017), 916--952.

\bibitem[OS]{OSbw}
S.~Ohta and K.-T.~Sturm, Bochner--Weitzenb\"ock formula and Li--Yau estimates on Finsler manifolds.
Adv.\ Math.\ {\bf 252} (2014), 429--448.

\bibitem[ON]{ON}
B.~O'Neill, Semi-Riemannian geometry: With applications to relativity.
Academic Press, Inc., New York, 1983.

\bibitem[Pen]{P}
R.~Penrose,
Gravitational collapse and space-time singularities.
Phys.\ Rev.\ Lett.\ {\bf 14} (1965), 57--59.

\bibitem[Per]{Per}
V.~Perlick, Fermat principle in Finsler spacetimes.
Gen.\ Relativ.\ Gravit.\ {\bf 38} (2006), 365--380.

\bibitem[vRS]{vRS}
M.-K.~von Renesse and K.-T.~Sturm,
Transport inequalities, gradient estimates, entropy and Ricci curvature.
Comm.\ Pure Appl.\ Math.\ {\bf 58} (2005), 923--940.

\bibitem[Sh]{Shlec}
Z.~Shen, Lectures on Finsler geometry.
World Scientific Publishing Co., Singapore, 2001.

\bibitem[St1]{StI}
K.-T.~Sturm, On the geometry of metric measure spaces.~I.
Acta Math.\ {\bf 196} (2006), 65--131.

\bibitem[St2]{StII}
K.-T.~Sturm, On the geometry of metric measure spaces.~II.
Acta Math.\ {\bf 196} (2006), 133--177.

\bibitem[Su]{Su}
S.~Suhr, Theory of optimal transport for Lorentzian cost functions.
M\"unster J.\ Math.\ {\bf 11} (2018), 13--47.

\bibitem[Vi]{Vi}
C.~Villani, Optimal transport, old and new.
Springer-Verlag, Berlin, 2009.

\bibitem[WW1]{WW1}
E.~Woolgar and W.~Wylie,
Cosmological singularity theorems and splitting theorems for $N$-Bakry--\'Emery spacetimes.
J.\ Math.\ Phys.\ {\bf 57} (2016), 022504, 1--12.

\bibitem[WW2]{WW2}
E.~Woolgar and W.~Wylie,
Curvature-dimension bounds for Lorentzian splitting theorems.
J.\ Geom.\ Phys.\ {\bf 132} (2018), 131--145.

\bibitem[Wy]{Wy}
W.~Wylie, A warped product version of the Cheeger--Gromoll splitting theorem.
Trans. Amer. Math. Soc. {\bf 369} (2017), 6661--6681.
\end{thebibliography}
\end{document}